\RequirePackage{amsthm}

\documentclass[spmpsco]{sn-jnl}

\usepackage{graphicx}%
\usepackage{multirow}%
\usepackage{amsmath}%
\usepackage{amssymb}%
\usepackage{amsfonts}%
\usepackage{booktabs}%
\usepackage{algorithm}%
\usepackage{algorithmic}
\usepackage{mathtools}%
\usepackage{comment}

\newcommand{\N}{\mathbb{N}}
\newcommand{\R}{\mathbb{R}}
\newcommand{\bx}{\boldsymbol{x}}

\DeclareMathOperator{\rk}{rank}
\DeclareMathOperator{\spn}{span}
\DeclareMathOperator{\Range}{Rg}
\DeclareMathOperator{\dom}{dom}

\theoremstyle{thmstyleone}%
\newtheorem{teo}{Theorem}
\newtheorem{prop}[teo]{Proposition}%
\newtheorem{coro}[teo]{Corollary}%

\theoremstyle{thmstyletwo}%
\newtheorem{ass}{Assumptions}%
\newtheorem{exmp}{Example}%
\newtheorem{rmk}{Remark}%

\theoremstyle{thmstylethree}%
\usepackage[style=numeric,backend=bibtex,giveninits=true,maxbibnames=300,sortcites=true]{biblatex} 
\begin{document}

\title[Article Title]{The iterated Golub-Kahan-Tikhonov method
\hfill\break {\small Dedicated to {\AA}ke Bj\"orck and 
Lars Eld\'en on the occasion of their 170th birthday.}}

\author[1]{\fnm{Davide} \sur{Bianchi}}\email{bianchid@mail.sysu.edu.cn}

\author[2]{\fnm{Marco} \sur{Donatelli}}\email{marco.donatelli@uninsubria.it}

\author[2]{\fnm{Davide} \sur{Furchì}}\email{dfurchi@uninsubria.it}

\author[3]{\fnm{Lothar} \sur{Reichel}}\email{reichel@math.kent.edu}

\affil[1]{\orgdiv{School of Mathematics (Zhuhai)}, \orgname{Sun Yat-sen University}, \orgaddress{\city{Zhuhai}, \postcode{519082}, \country{China}}}

\affil[2]{\orgdiv{Dipartimento di
		Scienza e Alta Tecnologia}, \orgname{Universit\`a dell'Insubria}, \orgaddress{\city{Como}, \postcode{22100}, \country{Italy}}}

\affil[3]{\orgdiv{Department of
		Mathematical Sciences}, \orgname{Kent State University}, \orgaddress{\city{Kent}, \postcode{OH 44242}, \country{USA}}}

\abstract{The Golub-Kahan-Tikhonov method is a popular solution technique for large linear discrete ill-posed problems. This method first applies partial Golub-Kahan bidiagonalization to reduce the size of the given problem and then uses Tikhonov regularization to compute a
meaningful approximate solution of the reduced problem. It is well known that iterated
variants of this method often yield approximate solutions of higher quality than the
standard non-iterated method. Moreover, it produces more accurate computed solutions than 
the Arnoldi method when the matrix that defines the linear discrete ill-posed problem is 
far from symmetric.

This paper starts with an ill-posed operator equation in infinite-dimensional Hilbert space, discretizes the equation, and then applies the iterated Golub-Kahan-Tikhonov method to the solution of the latter problem. An error analysis that addresses all discretization and approximation errors is provided. Additionally, a new approach for choosing the regularization parameter is described. This solution scheme 
produces more accurate approximate solutions than the standard (non-iterated) Golub-Kahan-Tikhonov method and the iterated Arnoldi-Tikhonov method.}


\maketitle


\section{Introduction}
Let $T\colon\mathcal{X}\to \mathcal{Y}$ denote a bounded linear operator between separable
Hilbert spaces $\mathcal{X}$ and $\mathcal{Y}$ with norms $\|\cdot\|_{\mathcal X}$ and
$\|\cdot\|_{\mathcal Y}$, respectively, that are induced by inner products. We are 
concerned with operators $T$ that are not continuously invertible. Such operators arise, 
for instance, from Fredholm integral equations of the first kind. 

We consider the solution of linear operator equations of the form
\begin{equation}\label{eq}
  Tx=y,\quad x\in\mathcal{X},\quad y\in\Range(Y),
\end{equation}
which we assume to be solvable, and denote the unique least-squares solution of minimal 
norm by $x^\dagger$. Since $T$ is not continuously invertible, $x^\dagger$ might not
depend continuously on $y$. The computation of the solution of \eqref{eq} therefore is an
\emph{ill-posed problem}.

In many problems of the form \eqref{eq} of interest in applications, the right-hand side 
is not known. Instead, only an error-contaminated approximation $y^{\delta}\in\mathcal{Y}$ 
of $y$ is available. This situation arises, for instance, when the available right-hand 
side is determined by measurements. We will assume that $y^{\delta}$ satisfies 
\begin{equation*}
    \|y-y^{\delta}\|_{\mathcal{Y}}\leq\delta
\end{equation*}
with a known bound $\delta>0$. 

The availability of $y^\delta$ instead of $y$ suggests that we should seek to compute an 
approximate solution $\widetilde{x}$ of the equation
\begin{equation}\label{eqdelta}
    Tx=y^{\delta},\quad x\in\mathcal{X},\quad y^{\delta}\in\mathcal{Y},
\end{equation}
such that $\widetilde{x}$ is an accurate approximation of $x^\dagger$. Here we tacitly 
assume that the equation \eqref{eqdelta} is solvable. Note that equations of the form 
\eqref{eqdelta} arise in many applications including remote sensing \cite{diazdealba2019}, 
atmospheric tomography \cite{ramlau2012}, computerized tomography \cite{natterer2001}, 
adaptive optics \cite{raffetseder2016}, and image restoration \cite{bentbib2018}.

Since $T$ is not continuously invertible, the least-squares solution of minimal 
norm of \eqref{eqdelta} typically is not a useful approximation of the solution 
$x^\dagger$ of \eqref{eq}. To be able to determine an accurate approximation of 
$x^\dagger$ from \eqref{eqdelta}, the latter equation has to be \emph{regularized}, i.e., 
the operator $T$ has to be replaced with a nearby operator so that the solution of the equation so obtained is less sensitive to the error in $y^\delta$ than the solution of
\eqref{eqdelta}. 

Much of the literature on the solution of operator equations \eqref{eqdelta} focuses on 
the analysis of these problems in infinite-dimensional Hilbert spaces, but properties of 
the finite-dimensional problems that arise through discretization before solution often 
are ignored. Conversely, many contributions to the literature on solution techniques 
for large-scale discretized problems obtained from \eqref{eqdelta} do not take into 
account the effects of the discretization error and of the approximation errors that 
stems from reducing the dimension of a large-scale discretized problem to a problem of 
smaller size. Among the exceptions are the papers \cite{bianchi2023itat,ramlau2019}, which
use results by Natterer \cite{natterer1977} and Neubauer \cite{neubauer1988a} to address 
this gap within the framework of Arnoldi-Tikhonov methods. These methods apply a few steps
of the Arnoldi process to reduce the generally large matrix that is obtained when 
discretizing equation \eqref{eqdelta} to a linear system of algebraic equations with a 
matrix of small size. Applications of the Arnoldi process to the solution of large linear
systems of equations that arise from the discretization of equations of the form~\eqref{eqdelta}, without discussing the effect of the discretization, are described in,
e.g., \cite{alkilayh2023arnoldi,calvetti2000tikhonov,gazzola2015krylov,lewis2009arnoldi,gn2016,n2018}.

The main attraction of the Arnoldi process, when compared to the Golub-Kahan process to be
discussed in this paper, is that the Arnoldi process does not require access to the 
transpose of the system matrix that is to be reduced, while the Golub-Kahan process does.
This feature of the Arnoldi process makes it possible to reduce matrices for which 
matrix-vector products are easy to evaluate, but for which matrix-vector products with the
transposed matrix are not. This situation may arise when the matrix is not explicitly 
formed, such as when the matrix is implicitly defined by a multipole method; see, e.g., 
\cite{greengard1997}. Moreover, the Arnoldi process may require fewer matrix-vector 
product evaluations than the Golub-Kahan reduction methods, because the latter requires
matrix-vector product evaluations with both the matrix and its transpose in every step; 
see \cite{calvetti2000} for an illustration. However, there are operator equations
\eqref{eqdelta} for which the Arnoldi process is known to furnish poor results, such as 
when the operator $T$ models motion blur and the desired solution $x^\dagger$ is the
blur- and noise-free image associated with an available blur- and noise-contaminated 
image $y^\delta$; see \cite{donatelli2015arnoldi} for a computed example in finite 
dimensions. For some image restoration problems, it is possible to mitigate this 
difficulty by preconditioning; see \cite{buccini2023arnoldi,gnnr}. Nevertheless, the difficulty 
of solving certain operator equations by reducing the associated discretized equation to a
problem of small size by the Arnoldi process makes the Golub-Kahan process the default 
method for reducing a large linear system of algebraic equations to a system of small
dimension; see, e.g., \cite{bentbib2018,bjorck88,bjorck24,calvetti2003,gazzola2015krylov} 
for discussions and illustrations of this method.

The Golub-Kahan bidiagonalization process can be applied in Hilbert space and bounds
for the error incurred by carrying out only a few steps with this process when applied to
the solution of \eqref{eqdelta} are described in \cite{alqahtani2023error}. This analysis
does not use results by Natterer \cite{natterer1977}. The present paper applies bounds 
derived by Natterer to estimate the influence of the discretization error. We consider 
iterated Tikhonov regularization based on partial Golub-Kahan bidiagonalization instead of
standard (non-iterated) Tikhonov regularization, because the former regularization method 
typically gives approximations of the desired solution $x^\dagger$ of \eqref{eq} of higher 
quality. Illustrations of the superior quality of approximate solutions computed by 
iterated Tikhonov methods applied to discretized problems can be found in, e.g.,
\cite{bianchi2015iterated,buccini2017iterated,buccini2020gsvd,buccini2023arnoldi,donatelli2012nondecreasing,
hanke1998nonstationary}.  When only one iteration of the iterated Golub-Kahan-Tikhonov 
method is carried out, the method reduces to the classical Golub-Kahan-Tikhonov method. The first appearances of this type of methods can be found in \cite{bjorck88,os1981,ps1982}, where the Golub-Kahan process is referred to as the Lanczos process.

It is the purpose of the present paper to analyze the effect of the discretization error
and the error that stems from the dimension reduction achieved with the Golub-Kahan 
process on the computed solution determined by iterated Tikhonov regularization for the 
approximate solution of equation \eqref{eqdelta}. Our analysis parallels that in 
\cite{bianchi2023itat}, which concerns iterated Tikhonov regularization based on partial 
Arnoldi decomposition of the operator $T$. As we will see, some of the bounds derived in 
\cite{bianchi2023itat} also apply to the iterated Golub-Kahan-Tikhonov method of the 
present paper. Our analysis, which is based on work by Neubauer \cite{neubauer1988a}, 
suggests a new approach to determining the regularization parameter. In particular, this approach differs from the discrepancy principle, which is described, e.g., in \cite{engl1996} as well as below.

This paper is organized as follows. Section~\ref{sec:pre} discusses the discretization of
equation \eqref{eqdelta} and introduces the iterated Tikhonov method. 
Section~\ref{sec:iAT} recalls the iterated Arnoldi-Tikhonov method defined in 
\cite{bianchi2023itat}.  The Golub-Kahan-Tikhonov method is reviewed in Section 
\ref{sec:GK}, and Section~\ref{sec:iGKT} describes our iterated Golub-Kahan-Tikhonov 
method and provides convergence results. An alternative parameter choice method is 
discussed in Section \ref{impsel} and a few computed examples are presented in 
Section~\ref{sec:comp}. Concluding remarks can be found in Section~\ref{sec:end}. An
appendix presents some of the theoretical results.

We conclude this section with some comments on the work by Bj\"orck and Eld\'en on
the Golub-Kahan-Tikhonov method. They were among the first to discuss this solution method for linear discrete ill-posed problems; see, e.g., \cite{bjorck88,bjorck14,bjorck24,elden77,elden82,elden90,elden05}.
Their work has had significant impact on later developments of solution methods for this kind of problems. 

\section{Preliminaries}\label{sec:pre}
Let $T^\dagger$ stand for the Moore-Penrose pseudo-inverse of the operator $T$ in 
\eqref{eq} with
\begin{equation*}
T^\dagger \colon \dom(T^\dagger)\subseteq \mathcal{Y} \to \mathcal{X}, \qquad 
\mbox{where } \dom(T^\dagger)= \Range(T)\oplus \Range(T)^\perp.
\end{equation*}
For any $y\in\dom(T^\dagger)$, the element $x^\dagger\coloneqq T^\dagger y$ is the unique
least-square solution of minimal norm of equation \eqref{eq}.

Since $T$ is not continuously invertible, the operator $T^\dagger$ is unbounded. This may
make the least-squares solution $T^\dagger y$ of \eqref{eq} very sensitive to the error in
$y^\delta$. A regularization method replaces $T^\dagger$ by a member of a family 
$\{R_\alpha \colon \mathcal{Y} \to \mathcal{X}\}$ of continuous operators that depend on a
parameter $\alpha$, paired with a suitable parameter choice rule 
$\alpha=\alpha(\delta, y^\delta)>0$. The pair $(R_\alpha,\alpha)$ furnishes a point-wise 
approximation of $T^\dagger$; see \cite[Definition 3.1]{engl1996} for a rigorous
definition.

\subsection{Discretization and Tikhonov regularization}\label{ssec:disc}
When to compute an approximate least-squares solution of \eqref{eqdelta}, we first 
discretize the equation and then compute an approximate solution of the discretized
equation obtained. The discretization introduces a discretization error. To bound the 
propagated discretization error, we use results by Natterer \cite{natterer1977} following 
a similar approach as in \cite{ramlau2019}.

Consider a sequence $\mathcal{X}_1 \subset \mathcal{X}_2 \subset \ldots \subset 
\mathcal{X}_n \subset \ldots \subset \mathcal{X}$ of finite-dimensional subspaces 
$\mathcal{X}_n$ of $\mathcal{X}$ with $\dim(\mathcal{X}_n)<\infty$, whose union is dense 
in $\mathcal{X}$. Define the projectors 
$P_n\colon\mathcal{\mathcal{X}}\to\mathcal{\mathcal{X}}_n$ and 
$Q_n\colon\mathcal{Y}\to \mathcal{Y}_n \coloneqq T(\mathcal{X}_n)$, as well as the 
inclusion operator $\iota_n \colon \mathcal{X}_n \hookrightarrow \mathcal{X}$. Application
of these operators to equations \eqref{eq} and \eqref{eqdelta} gives the equations
\begin{align*}
	&Q_nT\iota_n P_n x=Q_ny,\\
    &Q_nT\iota_nP_nx=Q_ny^{\delta}.
\end{align*}
Introduce the operator $T_n \colon \mathcal{X}_n \to \mathcal{Y}_n$,
$$
T_{n}\coloneqq Q_nT\iota_n,
$$
and the finite-dimensional vectors 
\begin{equation*}
y_n \coloneqq Q_ny,\quad y^\delta_n \coloneqq Q_ny^\delta,\quad x_n \coloneqq P_nx.
\end{equation*}
It is natural to identify $T_{n}$ with a matrix in $\R^{n_1\times n_2}$ with $n_2=\dim(\mathcal{X}_n)$, $n_1=\dim(\mathcal{Y}_n)$, and $y_n$, 
$y^\delta_n$ with elements in $\R^{n_1}$, and $x_n$ with an element in $\R^{n_2}$. This gives us the linear systems of 
equations
\begin{align}
	&T_{n}x_n=y_n,\label{disceq}\\
    &T_{n}x_n=y_n^{\delta}.\label{disceqdelta}
\end{align}
Henceforth, we will consider $T_n$ a matrix that represents a discretization of 
$T$. 

Let $T^{\dagger}_{n}$ denote the Moore-Penrose pseudo-inverse of the matrix $T_{n}$. Then 
the unique least-squares solutions with respect to the Euclidean vector norm  of equations
\eqref{disceq} and \eqref{disceqdelta} are given by
\begin{equation*}
  x_n^{\dagger}\coloneqq T^{\dagger}_{n}y_n\quad\mbox{and}\quad 
  x_n^{\dagger, \delta}\coloneqq T^{\dagger}_{n}y_n^{\delta},
\end{equation*}
respectively. The fact that the operator $T$ has an unbounded inverse results in that the
matrix $T_n$ is severely ill-conditioned and may be singular. It follows that the vector
$x_n^{\dagger,\delta}$, even when it is defined, typically is a useless approximation of $x_n^{\dagger}$. We conclude that 
regularization of the discretized operator equation \eqref{disceqdelta} is required.

Moreover, note that the solution $x^{\dagger}_n\in\mathcal{X}_n$ of \eqref{disceq} might not be an 
accurate approximation of the solution $x^\dagger$ of \eqref{eq}, due to a large 
propagated discretization error. We therefore are interested in determining a bound for 
$\|x^\dagger-x^{\dagger}_n\|_{\mathcal{X}}$. This, in general, requires some additional 
assumptions. In particular, it is not sufficient for $T$ and $T_{n}$ to be close in the 
operator norm; see \cite[Example 3.19]{engl1996}. We will assume that  
\begin{equation}\tag{H1}\label{hp:x_n-convergence}
\|x^\dagger-x^{\dagger}_n\|_{\mathcal{X}}\leq f(n)\rightarrow 0\quad\mbox{as } 
n\to\infty,
\end{equation}
for a suitable function $f$. For instance, if $T$ is compact and 
$\limsup_{n\to \infty}\|T_n^{\dagger *}x^\dagger_n\|_{\mathcal{X}}<\infty$, 
where the superscript $^*$ stands for
matrix transposition, or if $T$ is 
compact and $\{\mathcal{\mathcal{X}}_n\}_n$ is a discretization resulting from the 
\emph{dual-least square projection} method, see \cite[Section 3.3]{engl1996}, then 
\begin{equation*}
f(n)=O(\|(I-P_n)T^{\ast}\|),
\end{equation*}
where $\|\cdot\|$ denotes the operator norm induced by the norms $\|\cdot\|_{\mathcal X}$
and $\|\cdot\|_{\mathcal Y}$.


We conclude this subsection by letting $\lbrace e_j\rbrace_{j=1}^{n_2}$ be a convenient basis 
for $\mathcal{X}_n$. Consider the representation
\begin{equation*}
  x_n=\sum\limits_{j=1}^{n_2}x_j^{(n)}e_j
\end{equation*}
of an element $x_n\in\mathcal{\mathcal{X}}_n$ and identify $x_n$ with the vector 
\begin{equation*}
  \bx_n=[x_1^{(n)},\ldots,x_{n_1}^{(n)}]^{\ast}\in\R^{n_2}.
\end{equation*}
If $\lbrace e_j\rbrace_{j=1}^{n_2}$ is an orthonormal basis, then we may choose the norms so
that $\|\bx_n\|_2 = \|x_n\|_{\mathcal{X}}$. Here and throughout this paper $\|\cdot\|_2$
denotes the Euclidean vector norm. However, for certain discretization methods, the basis
$\lbrace e_j\rbrace_{j=1}^{n_2}$ is not orthonormal and this equality does not hold. Following
\cite{ramlau2019}, we then assume that there are positive constants $c_{\min}$ and 
$c_{\max}$, independent of $n$, such that 
\begin{equation}\tag{H2}\label{hp:norm_equivalence}
	c_{\min}\|\bx_n\|_2\leq\|x_n\|_\mathcal{X}\leq c_{\max}\|\bx_n\|_2.
\end{equation}
Such inequalities hold in many practical situations, e.g., when using B-splines, wavelets, 
and the discrete cosine transform as basis for $\mathcal{X}$; see \cite{de1978practical,goodman2016discrete}.

Due to the ill-conditioning of the matrix $T_n$ in \eqref{disceqdelta}, the least-squares 
solution of \eqref{disceqdelta} of minimal norm $x_n^{\dagger,\delta}$ generally is not a 
useful approximation of the minimal norm solution $x_n^\dagger$ of \eqref{disceq}.
Therefore, the operator $T_n^\dagger$ has to be replaced with a nearby operator that is 
less sensitive to the error in $y^\delta$. A common replacement is furnished by Tikhonov 
regularization which, when applied to \eqref{disceqdelta}, reads 
\begin{equation*}
x_{\alpha,n}^{\delta}\coloneqq \underset{x_n \in \R^{n_2}}{\operatorname{argmin}} 
\|T_nx_n-y_{n}^{\delta}\|_2^2+\alpha\|x_{n}\|_2^2.
\end{equation*}
The regularization parameter $\alpha>0$ determines the amount of regularization. The 
Tikhonov solution $x_{\alpha,n}^{\delta}$ can be expressed as
\begin{equation}\label{Tikhonov_n:closed_form}
	x_{\alpha,n}^{\delta}= (T_n^{\ast}T_n+\alpha I_n)^{-1}T_n^{\ast}y_n^\delta,
\end{equation}
where $I_n$ denotes the identity matrix of order $n$.

A generalized version of this method is the iterated Tikhonov method, which is known to 
give a better approximation of the desired solution $x^\dagger$. The closed form formula
for the $i$-th approximation given by this method is
\begin{equation}\label{itTikhonov_n:closed_form}
x_{\alpha,n,i}^{\delta}=\sum\limits_{k=1}^i\alpha^{k-1}(T_{n}^{\ast}
T_n+\alpha I_n)^{-k}T_{n}^{\ast}y_n^{\delta}.
\end{equation}
Thus, the Tikhonov solution \eqref{Tikhonov_n:closed_form} corresponds to $i=1$.

Clearly, when $n$ is large, computing  $x_{\alpha,n,i}^{\delta}$ by formula \eqref{itTikhonov_n:closed_form}, or $x_{\alpha,n}^{\delta} = x_{\alpha,n,1}^{\delta}$ by 
formula~\eqref{Tikhonov_n:closed_form}, is impractical. This work discusses how to reduce the complexity of iterated Tikhonov regularization and achieve a fairly accurate approximation of
$x_{\alpha,n,i}^{\delta}$ by applying a few steps of Golub-Kahan bidiagonalization to the 
matrix $T_n$ with initial vector $y_n^\delta$. 


\section{The iterated Arnoldi-Tikhonov method}\label{sec:iAT}
We briefly recall the iterated Arnoldi-Tikhonov (iAT) method presented in 
\cite{bianchi2023itat}, which uses a partial Arnoldi decomposition instead of a partial 
Golub-Kahan bidiagonalization to implement the iterated Tikhonov method 
\eqref{itTikhonov_n:closed_form}, because some results for the iAT method carry over to 
the iterated Golub-Kahan Tikhonov (iGKT) method and because we will compare the 
performances of these methods. We assume in this section that $n=n_1=n_2$.

Application of $1\leq \ell\ll n$ steps of the Arnoldi process to the square matrix $T_{n}$ 
with initial vector $y_n^{\delta}$ generically gives the Arnoldi decomposition
\begin{equation*}
	T_{n}V_{n,\ell}=V_{n,\ell+1}H_{\ell+1,\ell}.
\end{equation*}
The columns of the matrix
\begin{equation*}
V_{n,\ell+1}=\left[\begin{array}{@{}c|c|c@{}}
	v_{n,\ell+1}^{(1)} &\cdots & v_{n,\ell+1}^{(\ell+1)}
\end{array}\right]=
\left[\begin{array}{@{}c|c@{}}
	V_{n,\ell} & v_{n,\ell+1}^{(\ell+1)}
\end{array}\right]\in\R^{n_1\times (\ell+1)}
\end{equation*}
form an orthonormal basis for the Krylov subspace
\begin{equation*}
\mathcal{K}_{\ell+1}(T_{n},y_n^{\delta})=
\spn\lbrace y_n^{\delta},T_{n}y_n^{\delta},\ldots,T_{n}^{\ell}y_n^{\delta}\rbrace.
\end{equation*}
Moreover, 
$H_{\ell+1,\ell}\in\R^{(\ell+1)\times\ell}$ is an upper Hessenberg matrix, i.e., all
entries below the subdiagonal vanish. We will assume the generic situation that all
subdiagonal entries of $H_{\ell+1,\ell}$ are nonvanishing. This requirement can 
easily be removed.

Define the Arnoldi approximation
\begin{equation}\label{T_n-a}
\tilde{T}_n^{(\ell)}\coloneqq V_{n,\ell+1}H_{\ell+1,\ell}V_{n,\ell}^{\ast} \in 
\R^{n\times n},
\end{equation}
of the matrix $T_n$. We will assume that
\begin{equation*}
\|T_n-\tilde{T}_n^{(\ell)}\|_2\leq \tilde{h}_{\ell} \quad \mbox{for some} 
\quad \tilde{h}_{\ell}\geq0.
\end{equation*}
Notice that when $\ell=n$, we have $\tilde{T}_n^{(n)} = T_n$.

We apply the iterated Tikhonov regularization \eqref{itTikhonov_n:closed_form} with the 
approximation \eqref{T_n-a}. Convergence results obtained in \cite{bianchi2023itat} are 
analogous to the ones in Section~\ref{sec:iGKT}. We summarize the iAT method in Algorithm 
\ref{algo:iAT}.

\begin{algorithm}
\caption{The iAT method}\label{algo:iAT}
\textbf{Input:} $\{T_n$, $y^\delta_n$, $\ell$, $i\}$\\
\textbf{Output:} $\tilde{x}_{\tilde{\alpha},n,i}^{\delta,\ell}$
\begin{algorithmic}[1]
\STATE Compute $\{V_{n,\ell+1}$, $H_{\ell+1,\ell}\}$ with the Arnoldi process 
\cite[Section 6.3]{saad2003}
\STATE Compute $\tilde{y}^\delta_{\ell+1}= V^{\ast}_{n,\ell+1}y^\delta_n$ 
\STATE Set $\tilde{\alpha}$
\STATE Compute $\tilde{z}_{\tilde{\alpha},\ell,i}^{\delta,\ell}=\sum\limits_{k=1}^i\tilde{\alpha}^{k-1}
(H_{\ell+1,\ell}^{\ast}H_{\ell+1,\ell}+\tilde{\alpha} I_{\ell})^{-k}H_{\ell+1,\ell}^{\ast}
\tilde{y}_{\ell+1}^{\delta}$
\STATE Return $\tilde{x}_{\tilde{\alpha},n,i}^{\delta,\ell}=V_{n,\ell}\tilde{z}_{\tilde{\alpha},\ell,i}^{\delta,\ell}$
\end{algorithmic}
\end{algorithm}

In order to compute the parameter $\tilde{\alpha}$ as proposed in \cite{bianchi2023itat}, 
we define the operator $\tilde{\mathcal{R}}_{\ell}$ to be the orthogonal projector from 
$\R^{n}$ into $\Range(\tilde{T}_n^{(\ell)})$. Let $\tilde{q}=\rk(H_{\ell+1,\ell})$ and introduce the singular value decomposition 
\begin{equation*}
  H_{\ell+1,\ell}=\tilde{W}_{\ell+1}\tilde{\Sigma}_{\ell+1,\ell}\tilde{S}^{\ast}_{\ell},
\end{equation*}
where the matrices $\tilde{W}_{\ell+1}\in\R^{(\ell+1)\times(\ell+1)}$ and 
$\tilde{S}_{\ell}\in\R^{\ell\times\ell}$ are orthogonal, and the diagonal entries of the 
matrix
\begin{equation*}
\tilde{\Sigma}_{\ell+1,\ell}=
{\rm diag}[\tilde{\sigma}_1,\tilde{\sigma}_2,\ldots,\tilde{\sigma}_\ell]\in 
\R^{(\ell+1)\times\ell}
\end{equation*}
are ordered according to $\tilde{\sigma}_1\geq\ldots\geq\tilde{\sigma}_{\tilde{q}}>
\tilde{\sigma}_{\tilde{q}+1}=\ldots=\tilde{\sigma}_\ell=0$. Since all subdiagonal entries
of the matrix $H_{\ell+1,\ell}$ are nonvanishing, the matrix has full rank 
$\tilde{q}=\ell$. Let
\begin{equation*}
I_{\ell,\ell+1}=\begin{bmatrix}
  I_{\ell} & 0\\
  0 & 0
\end{bmatrix}\in\R^{(\ell+1)\times(\ell+1)},
\end{equation*}
then
\begin{equation*}
\tilde{\mathcal{R}}_{\ell}=V_{n,\ell+1}\tilde{W}_{\ell+1}I_{\ell,\ell+1}\tilde{W}_{\ell+1}^{\ast}V_{n,\ell+1}^{\ast}.
\end{equation*}
Define 
\begin{equation*}
  \tilde{\hat{y}}^{\delta}_{\ell+1}\coloneqq I_{\ell,\ell+1}\tilde{W}_{\ell+1}^{\ast}\tilde{y}^{\delta}_{\ell+1}
\end{equation*}
and assume that at least one of the first $\tilde{q}$ entries of the vector $\tilde{\hat{y}}^{\delta}_{\ell+1}$ is non-vanishing. Then the equation
\begin{equation}\label{conda}
\tilde{\alpha}^{2i+1}(\tilde{\hat{y}}_{\ell+1}^\delta)^{\ast}(\tilde{\Sigma}_{\ell+1,\ell}
\tilde{\Sigma}_{\ell+1,\ell}^{\ast}+\tilde{\alpha} I_{\ell+1})^{-2i-1}\tilde{\hat{y}}_{\ell+1}^\delta=
(E\tilde{h}_{\ell}+C\delta)^2,
\end{equation}
possesses a unique solution $\tilde{\alpha}>0$ if we choose positive constants $C$ and $E$
such that
\begin{equation}\label{condECa}
 0\leq E\tilde{h}_{\ell}+C\delta\leq\|\tilde{\mathcal{R}}_{\ell}\tilde{y}_n^{\delta}\|_2=
 \|I_{\ell,\ell+1}\tilde{W}_{\ell+1}^{\ast}\tilde{y}_{\ell+1}^{\delta}\|_2.
\end{equation}

We recall the the results of \cite{bianchi2023itat}.

\begin{prop}\label{comp_iAT}
\cite[Proposition 1]{bianchi2023itat} Set $C=1$ and $E=\|x_n^{\dagger}\|_2$ in equation 
\eqref{conda}. Let equation \eqref{condECa} hold and let $\tilde{\alpha}>0$ be the 
unique solution of \eqref{conda}. Then for all $\hat{\alpha}\geq\tilde{\alpha}$, we have that 
$\|x_n^{\dagger}-\tilde{x}_{\tilde{\alpha},n,i}^{\delta,\ell}\|_2\leq\|x_n^{\dagger}-
\tilde{x}_{\hat{\alpha},n,i}^{\delta,\ell}\|_2$.
\end{prop}

\begin{prop}\label{convA}
\cite[Proposition $2$]{bianchi2023itat} Set $C=1$ and $E=\|x^{\dagger}_n\|_2$ in equation \eqref{conda}. Let equation \eqref{condECa} hold and let 
$\tilde{\alpha}>0$ be the unique solution of \eqref{conda}. For some $\nu\geq 0$ and $\rho>0$, let
$x_n^{\dagger}\in\mathcal{X}_{n,\nu,\rho}$, where 
\begin{equation*}
\mathcal{X}_{n,\nu,\rho}\coloneqq\lbrace x_n\in\mathcal{X}_n \mid 
x_n=(T_n^{\ast}T_n)^{\nu}w_n,\; w_n\in \ker(T_n)^\perp \mbox{ and } 
\|w_n\|_2\leq\rho\rbrace.
\end{equation*}
Then
\begin{equation*}
\|x_n^{\dagger}-\tilde{x}_{\tilde{\alpha},n,i}^{\delta,\ell}\|_2=
\begin{cases}
 o(1) &\text{if}\quad\nu=0,\\
 o((h_{\ell}+\delta)^{\frac{2\nu i}{2\nu i+1}})+O(\gamma_{\ell}^{2\nu}\|w_n\|_2)&\text{if}
 \quad 0<\nu<1,\\
 O((h_{\ell}+\delta)^{\frac{2i}{2i+1}})+O(\gamma_{\ell}
 \|(I_n-\tilde{\mathcal{R}}_{\ell})T_nw_n\|_2) &\text{if}\quad\nu=1,
 \end{cases}
\end{equation*}
where $\gamma_{\ell}\coloneqq \|(I_n-\tilde{\mathcal{R}}_{\ell})T_n\|_2$.
\end{prop}

Observe that, since for $\ell$ large enough it holds $h_{\ell}=0$, condition \eqref{condECa} with $C=1$ is generically satisfied for the Krylov subspace. A similar condition for the Golub-Kahan method will behave in the same way.

\begin{coro}\label{convtotA}
\cite[Corollary 3]{bianchi2023itat} Assume that $x_n^{\dagger}\in\mathcal{X}_{n,1,\rho}$ and let $\tilde{\alpha}>0$ be the solution of
\eqref{conda}. Then for $\ell$ such that $\tilde{h}_{\ell}\sim\delta$, we have
\begin{eqnarray*}
  \|x_n^{\dagger}-\tilde{x}^{\delta,\ell}_{\tilde{\alpha},n,i}\|_2&=&O(\delta^{\frac{2i}{2i+1}})\qquad
  \mbox{as }\delta\rightarrow 0, \\
  \|x^{\dagger}-\tilde{x}_{\tilde{\alpha},n,i}^{\delta,\ell}\|_{\mathcal{X}}&\leq& f(n)+
  O(\delta^{\frac{2i}{2i+1}})\qquad\mbox{as }\delta\rightarrow 0.
\end{eqnarray*}
\end{coro}

As discussed at the end of \cite[Appendix~A.2]{bianchi2023itat}, if an estimate of 
$\|x_n^{\dagger}\|_2$ is not available, then we may substitute $E$ by the expression 
$D\|\tilde{x}_{\tilde{\alpha},n,i}^{\delta,\ell}\|_2$ for some constant $D\geq 1$. With this choice, for 
$\tilde{\alpha}$ satisfying \eqref{condECa}, we achieve the same convergence rates.


\section{The Golub-Kahan-Tikhonov method}\label{sec:GK}
Golub-Kahan bidiagonalization is a commonly used technique to reduce a large matrix to a 
small one, while retaining some of the important characteristics of the large matrix. This 
section reviews the Golub-Kahan bidiagonalization method and shows how it can be applied 
to approximate the discretized equation \eqref{disceqdelta} and its Tikhonov regularized 
solution \eqref{Tikhonov_n:closed_form} in a low-dimensional subspace $\R^\ell$ with 
$1\leq \ell\ll\min\lbrace n_1,n_2\rbrace$. Thorough discussions of the Golub-Kahan 
decomposition can be found in Bj\"orck~\cite{bjorck24} and in Golub and Van Loan 
\cite{golub2013}. Applications to Tikhonov regularization are discussed, e.g., in
\cite{bjorck88,calvetti2003,gazzola2015krylov}.

\begin{algorithm}[H]
\caption{Golub-Kahan bidiagonalization}\label{algo:GK}
\textbf{Input:} matrix $T_n\in\R^{n_1\times n_2}$, number of steps 
$1\leq\ell\ll\min\lbrace n_1,n_2\rbrace$, initial vector $y_n^\delta\in\R^{n_2}$\\
\textbf{Output:} lower bidiagonal matrix $B_{\ell+1,\ell}\in\R^{(\ell+1)\times\ell}$, 
matrices \\ $U_{\ell+1}=[u_1,u_2,\ldots,u_{\ell+1}]\in\R^{n_1\times(\ell+1)}$ and 
$V_{\ell}=[v_1,v_2,\ldots,v_{\ell}]\in\R^{n_2\times\ell}$ with \\ orthonormal columns
\begin{algorithmic}[1]
\STATE $\beta_1=\|y_n^\delta\|_2$; $u_1=y_n^\delta/\beta_1$; $v_0=0$;
\FOR {$i=1\; \text{to}\; \ell$}
\STATE $w=T_n^* u_i-\beta_iv_{i-1}$; 
\STATE $\alpha_i=\|w\|_2$; $v_i=w/\alpha_i$; $B_{i,i}=\alpha_i$
  \STATE $w=T_n v_i - \alpha_i u_i$
  \STATE $\beta_{i+1}=\|w\|_2$; $u_{i+1}=w/\beta_{i+1}$; $B_{i+1,i}=\beta_{i+1}$
\ENDFOR
\end{algorithmic}
\end{algorithm}

Generically, Algorithm \ref{algo:GK} determines the Golub-Kahan decompositions
\begin{equation}\label{gk}
T_nV_\ell=U_{\ell+1}B_{\ell+1,\ell}\quad\mbox{and}\quad 
T_n^*U_\ell=V_\ell B_{\ell,\ell}^*,
\end{equation}
where $V_{\ell}\in\R^{n_2\times\ell}$, the matrix $U_\ell\in\R^{n_1\times\ell}$ is made up
of the first $\ell$ columns of $U_{\ell+1}\in\R^{n_1\times(\ell+1)}$ and the matrix 
$B_{\ell,\ell}\in\R^{\ell\times\ell}$ consists of the first $\ell$ rows of the lower 
bidiagonal matrix
\begin{equation*}
B_{\ell+1,\ell}=\left[\begin{array}{ccccc}
\alpha_1 & 0 & & \cdots  & 0 \\
\beta_2  & \alpha_2 & 0 & \cdots  & \vdots \\
    0    & \beta_3 & \ddots &  & \\
    \vdots  &         & \ddots & \alpha_{\ell-1} & 0\\
         &         &        & \beta_{\ell} & \alpha_\ell \\
0         & \cdots        &        &  & \beta_{\ell+1} 
\end{array}\right]\in\R^{(\ell+1)\times\ell}.
\end{equation*}
We will assume the generic situation that all nontrivial entries of $B_{\ell+1,\ell}$ are 
nonvanishing.

Moreover, we have
\begin{eqnarray*}
{\rm span}\{u_1,u_2,\ldots,u_{\ell+1}\}&=&
{\rm span}\{y^\delta_n,T_nT_n^*y^\delta_n,\ldots,(T_nT_n^*)^\ell y_n^\delta\},\\
{\rm span}\{v_1,v_2,\ldots,v_{\ell}\}&=&
{\rm span}\{T_n^*y^\delta_n,(T_n^*T_n)T_n^*y^\delta_n,\ldots,
(T_nT_n^*)^{\ell-1}T_n^*y_n^\delta\}.
\end{eqnarray*}

Define the Golub-Kahan approximation 
\begin{equation}\label{T_n-gk}
 T_n^{(\ell)}\coloneqq U_{\ell+1}B_{\ell+1,\ell}V_{\ell}^{\ast} \in \R^{n_1\times n_2}
\end{equation}
of the matrix $T_n$. We will assume that
\begin{equation*}\label{hell}
	\|T_n-T_n^{(\ell)}\|_2\leq h_{\ell} \quad \mbox{for some} \quad h_{\ell}\geq0.
\end{equation*}
Note that when $\ell=\min\lbrace n_1,n_2\rbrace$ in Algorithm \ref{algo:GK}, we have 
$T_n^{(\ell)} = T_n$.

Having the Golub-Kahan approximation \eqref{T_n-gk} makes it possible to solve the 
\emph{approximated} discretized equation
\begin{equation}\label{eq:approximated_discretized_equation}
  T_n^{(\ell)} x_n = y_n^\delta
\end{equation}
instead of the discretized equation \eqref{disceqdelta}. This is beneficial because a
solution of equation~\eqref{eq:approximated_discretized_equation} can be computed much 
faster than a solution of \eqref{disceqdelta} when the matrix $T_n$ is large.

Application of Tikhonov regularization \eqref{Tikhonov_n:closed_form} to equation 
\eqref{eq:approximated_discretized_equation} gives the solution
\begin{equation}\label{gktm}
x_{\alpha,n}^{\delta,\ell}\coloneqq (T_n^{(\ell)*}T_n^{(\ell)}+\alpha I_n)^{-1}
T_n^{(\ell)*}y_n^{\delta}.
\end{equation}
We exploit the structure of $T_n^{(\ell)}$ to reduce the computational complexity of 
computing the solution \eqref{gktm}. For any $x_\ell \in \R^\ell$, we have
\begin{align*}
V_\ell (B^{\ast}_{\ell+1,\ell}B_{\ell+1,\ell}+\alpha I_{\ell})x_\ell &=
(V_\ell B^{\ast}_{\ell+1,\ell}U^{\ast}_{\ell+1}U_{\ell+1}B_{\ell+1,\ell}
V^{\ast}_\ell+\alpha I_{n})V_\ell x_\ell\\
&= (T_{n}^{(\ell)*}T_{n}^{(\ell)}+\alpha I_n)V_{\ell}x_\ell
\end{align*}
and, therefore,
\begin{equation}\label{HV}
(T_{n}^{(\ell)*}T_{n}^{(\ell)}+\alpha I)^{-1}V_{\ell}=V_{\ell}(B^{\ast}_{\ell+1,\ell}
B_{\ell+1,\ell}+\alpha I_{\ell})^{-1}.
\end{equation}
Let $y_{\ell+1}^\delta\coloneqq U_{\ell+1}^{\ast}y_n^\delta\in\R^{\ell+1}$. Then
\begin{equation*}
z_{\alpha,\ell}^{\delta,\ell}\coloneqq (B^{\ast}_{\ell+1,\ell}B_{\ell+1,\ell}+
\alpha I_{\ell})^{-1}B^{\ast}_{\ell+1,\ell}y_{\ell+1}^\delta 
\end{equation*}
is the Tikhonov regularized solution associated with the \emph{reduced} discretized 
equation
\begin{equation*}
    B_{\ell+1,\ell}z_\ell = y_{\ell+1}^\delta.
\end{equation*}
Combining equations \eqref{T_n-gk}, \eqref{gktm}, and \eqref{HV}, we obtain
\begin{equation*}
   x_{\alpha,n}^{\delta,\ell} =  V_{\ell}z_{\alpha,\ell}^{\delta,\ell}.
\end{equation*}
This procedure of computing an 
approximate solution of \eqref{disceqdelta}, and therefore of \eqref{eqdelta}, is 
referred to as the \emph{Golub-Kahan-Tikhonov} (GKT) \emph{regularization method}. 
Algorithm \ref{algo:GKT} summarizes the computations. We will comment on the choice
of the regularization parameter $\alpha>0$ below.

\begin{algorithm}
\caption{The Golub-Kahan-Tikhonov regularization method}\label{algo:GKT}
\textbf{Input:} $T_n$, $y^\delta_n$, $\ell$\\
\textbf{Output:} $x_{\alpha,n}^{\delta,\ell}$
\begin{algorithmic}[1]
\STATE Compute the matrices $U_{\ell+1}$, $V_\ell$, and $B_{\ell+1,\ell}$ with
Algorithm \ref{algo:GK}
\STATE Compute $y^\delta_{\ell+1}= U^{\ast}_{\ell+1}y^\delta_n$ 
\STATE Set $\alpha$
\STATE Compute $z_{\alpha,\ell}^{\delta,\ell}= (B^{\ast}_{\ell+1,\ell}B_{\ell+1,\ell}+
\alpha I_{\ell})^{-1}B^{\ast}_{\ell+1,\ell}y_{\ell+1}^\delta$
\STATE Compute $x_{\alpha,n}^{\delta,\ell}=V_\ell z_{\alpha,\ell}^{\delta,\ell}$
\end{algorithmic}
\end{algorithm}

A convergence analysis for approximate solutions of \eqref{eqdelta} computed with 
Algorithm~\ref{algo:GKT} is carried out in \cite{alqahtani2023error}, where also 
convergence rates for $\|x_n^\dagger - x_{\alpha,n}^{\delta,\ell}\|_2$ are established. 
Convergence in the space $\mathcal{X}$ then is obtained by using 
\eqref{hp:norm_equivalence}.


\section{The iterated Golub-Kahan-Tikhonov method}\label{sec:iGKT}
This section describes an iterated variant of the Golub-Kahan-Tikhonov regularization 
method outlined by Algorithm \ref{algo:GKT}. We refer to this variant as the iterated 
Golub-Kahan-Tikhonov (iGKT) method and analyze its convergence properties. Our interest in
the iGKT method stems from the fact that it delivers more accurate approximations of 
$x^\dagger$ than the GKT method. A reason for this is that the standard (non-iterated) 
Tikhonov regularization method exhibits a saturation rate of $O(\delta^{2/3})$ as 
$\delta\searrow 0$, as shown in \cite[Proposition 5.3]{engl1996}. It follows that 
approximate solutions computed with the GKT method do not converge to the solution of 
\eqref{eqdelta} faster than $O(\delta^{2/3})$ as $\delta\searrow 0$. However, the iGKT 
method can surpass this saturation rate, as is shown in Corollary \ref{convtotGK} below. 
We remark that the computational effort required by the iGKT method is essentially the 
same as for the non-iterated GKT method.

\subsection{Definition of the iGKT method}
The GKT method is combined with iterated Tikhonov regularization 
\eqref{itTikhonov_n:closed_form} using the matrix $T_n^{(\ell)}$ in~\eqref{T_n-gk}. This 
yields the iterative method
\begin{equation}\label{iGKTiter}
x_{\alpha,n,i}^{\delta,\ell}=\sum\limits_{k=1}^i\alpha^{k-1}(T_{n}^{(\ell)\ast}
T_n^{(\ell)}+\alpha I_n)^{-k}T_{n}^{(\ell)\ast}y_n^{\delta}.
\end{equation}
We will denote the left-hand side by $x_{\alpha,n,i}^{\ell}$ when the vector $y_n^\delta$
is replaced with $y_n$. Similarly as in Algorithm \ref{algo:GKT}, it 
follows from the Golub-Kahan decomposition that
\begin{equation*}
  x_{\alpha,n,i}^{\delta,\ell}=V_{\ell}z_{\alpha,\ell,i}^{\delta,\ell},
\end{equation*}
where
\begin{equation*}
  z_{\alpha,\ell,i}^{\delta,\ell}\coloneqq\sum_{k=1}^i\alpha^{k-1}(B_{\ell+1,\ell}^{\ast}
  B_{\ell+1,\ell}+\alpha I_{\ell})^{-k}B_{\ell+1,\ell}^{\ast}y_{\ell+1}^{\delta}.
\end{equation*}
The iGKT method is described by Algorithm \ref{algo:iGKT}. Step 4 of the algorithm is evaluated by using an iteration analogous to \eqref{iGKTiter} and computing the Cholesky factorization of the matrix 
$B_{\ell+1,\ell}^{\ast} B_{\ell+1,\ell}+\alpha I_{\ell}$. Notice that for $i=1$, we 
recover the GKT method of Algorithm \ref{algo:GKT}. Assume that the matrix $T_n$ is large 
and that $\ell\ll\min\lbrace n_1,n_2\rbrace$. This is a situation of interest to us. Then the main computational 
effort required by Algorithm~\ref{algo:iGKT} is the evaluation of $\ell$ 
matrix-vector products with each one of the matrices $T_n$ and $T^*_n$ needed to calculate
the Golub-Kahan decomposition \eqref{gk}. Note that the computational effort required by
Algorithm \ref{algo:iGKT} is essentially independent of the number of iterations $i$ 
of the iGKT method.

\begin{algorithm}
\caption{The iGKT method}\label{algo:iGKT}
\textbf{Input:} $T_n$, $y^\delta_n$, $\ell$, $i$\\
\textbf{Output:} $x_{\alpha,n,i}^{\delta,\ell}$
\begin{algorithmic}[1]
\STATE Compute the matrices $U_{\ell+1}$, $V_\ell$, and $B_{\ell+1,\ell}$ with 
Algorithm \ref{algo:GK} 
\STATE Compute $y^\delta_{\ell+1}= U^{\ast}_{\ell+1}y^\delta_n$ 
\STATE Set $\alpha$
\STATE Compute $z_{\alpha,\ell,i}^{\delta,\ell}=\sum\limits_{k=1}^i\alpha^{k-1}
(B_{\ell+1,\ell}^{\ast}B_{\ell+1,\ell}+\alpha I_{\ell})^{-k}B_{\ell+1,\ell}^{\ast}
y_{\ell+1}^{\delta}$
\STATE Compute $x_{\alpha,n,i}^{\delta,\ell}=V_{\ell}z_{\alpha,\ell,i}^{\delta,\ell}$
\end{algorithmic}
\end{algorithm}

\subsection{Convergence results}
This section discusses some convergence results for the iGKT method described by 
Algorithm~\ref{algo:iGKT}. The results parallel those shown for the iAT method in 
\cite{bianchi2023itat}. They only differ in that, in the present paper, we define the 
Golub-Kahan approximation \eqref{T_n-gk} of the matrix $T_n$, while in 
\cite{bianchi2023itat} the analogous approximation \eqref{T_n-a} is defined by using a 
partial Arnoldi decomposition of $T_n$. In particular, the proofs presented in 
\cite{bianchi2023itat} carry over to the setting of the present paper. Therefore, when
presenting the convergence analysis, we refer to \cite{bianchi2023itat} for proofs. 
Our results are based on work by Neubauer \cite{neubauer1988a}. This connection is clear
in the Appendix of \cite{bianchi2023itat}.

We will need the orthogonal projector $\mathcal{R}_{\ell}$ from $\R^{n_1}$ into 
$\Range(T_n^{(\ell)})$. Let $q = \operatorname{rank}(B_{\ell+1,\ell})$ and introduce the 
singular value decomposition 
\begin{equation*}
B_{\ell+1,\ell}=W_{\ell+1}\Sigma_{\ell+1,\ell}S_{\ell}^{\ast},
\end{equation*} 
where the matrices $W_{\ell+1}\in\R^{(\ell+1)\times(\ell+1)}$ and 
$S_\ell\in\R^{\ell\times\ell}$ are orthogonal, and the diagonal entries of the matrix
\begin{equation*}
\Sigma_{\ell+1,\ell}={\rm diag}[\sigma_1,\sigma_2,\ldots,\sigma_\ell]\in
\R^{(\ell+1)\times\ell}
\end{equation*}
are ordered according to 
$\sigma_1\geq\ldots\geq\sigma_q>\sigma_{q+1}=\ldots=\sigma_\ell=0$. Since we assume the Golub-Kahan process does not break down, the matrix $B_{\ell+1,\ell}$ has full rank $q=\ell$. Let
\begin{equation*}
I_{\ell,\ell+1}=\begin{bmatrix}
  I_{\ell} & 0\\
  0 & 0
  \end{bmatrix}\in\R^{(\ell+1)\times(\ell+1)}.
\end{equation*}
Then
\begin{equation*}
\mathcal{R}_\ell=U_{\ell+1}W_{\ell+1}I_{\ell,\ell+1}W_{\ell+1}^*U_{\ell+1}^*.
\end{equation*}

Define 
\begin{equation*}
\hat{y}_{\ell+1}^\delta \coloneqq I_{\ell,\ell+1}W_{\ell+1}^{\ast}y_{\ell+1}^\delta
\end{equation*}
and assume that at least one of the first $q$ entries of the vector 
$\hat{y}_{\ell+1}^\delta$ is non-vanishing. Then the equation
\begin{equation}\label{condgk}
\alpha^{2i+1}(\hat{y}_{\ell+1}^\delta)^{\ast}(\Sigma_{\ell+1,\ell}
\Sigma_{\ell+1,\ell}^{\ast}+\alpha I_{\ell+1})^{-2i-1}\hat{y}_{\ell+1}^\delta=
(Eh_{\ell}+C\delta)^2,
\end{equation}
with positive constants $C$ and $E$ has a unique solution $\alpha>0$ if we choose $C$
and $E$ so that
\begin{equation}\label{condECgk}
 0\leq Eh_{\ell}+C\delta\leq\|\mathcal{R}_{\ell}y_n^{\delta}\|_2=
 \|I_{\ell,\ell+1}W_{\ell+1}^{\ast}y_{\ell+1}^{\delta}\|_2.
\end{equation}

By using the same techniques as for Propositions~\ref{comp_iAT} and \ref{convA} and Corollary~\ref{convtotA} respectively, we obtain the following results.

\begin{prop}\label{comp_iGKT}
Set $C=1$ and $E=\|x_n^{\dagger}\|_2$ in equation \eqref{condgk}. Let equation 
\eqref{condECgk} hold and let $\alpha>0$ be the unique solution of \eqref{condgk}. Then 
for all $\hat{\alpha}\geq\alpha$, we have that 
$\|x_n^{\dagger}-x_{\alpha,n,i}^{\delta,\ell}\|_2\leq\|x_n^{\dagger}-
x_{\hat{\alpha},n,i}^{\delta,\ell}\|_2$.
\end{prop}

\begin{prop}\label{convGKT}
Set $C=1$ and $E=\|x^{\dagger}_n\|_2$. Assume that \eqref{condECgk} holds and let 
$\alpha>0$ be the unique solution of \eqref{condgk}. For some $\nu\geq 0$ and $\rho>0$, let
$x_n^{\dagger}\in\mathcal{X}_{n,\nu,\rho}$, where 
\begin{equation*}
\mathcal{X}_{n,\nu,\rho}\coloneqq\lbrace x_n\in\mathcal{X}_n \mid 
x_n=(T_n^{\ast}T_n)^{\nu}w_n,\; w_n\in \ker(T_n)^\perp \mbox{ and } 
\|w_n\|_2\leq\rho\rbrace.
\end{equation*}
Then
\begin{equation*}
\|x_n^{\dagger}-x_{\alpha,n,i}^{\delta,\ell}\|_2=
\begin{cases}
 o(1) &\text{if}\quad\nu=0,\\
 o((h_{\ell}+\delta)^{\frac{2\nu i}{2\nu i+1}})+O(\gamma_{\ell}^{2\nu}\|w_n\|_2)&\text{if}
 \quad 0<\nu<1,\\
 O((h_{\ell}+\delta)^{\frac{2i}{2i+1}})+O(\gamma_{\ell}
 \|(I_n-\mathcal{R}_{\ell})T_nw_n\|_2) &\text{if}\quad\nu=1,
 \end{cases}
\end{equation*}
where $\gamma_{\ell}\coloneqq \|(I_n-\mathcal{R}_{\ell})T_n\|_2$.
\end{prop}
    
\begin{coro}\label{convtotGK}
Assume that $x_n^{\dagger}\in\mathcal{X}_{n,1,\rho}$ and let $\alpha>0$ be the solution of
\eqref{condgk}. Then, for $\ell$ such that $h_{\ell}\sim\delta$, we have
\begin{eqnarray*}
  \|x_n^{\dagger}-x^{\delta,\ell}_{\alpha,n,i}\|_2&=&O(\delta^{\frac{2i}{2i+1}})\qquad
  \mbox{as }\delta\rightarrow 0, \\
  \|x^{\dagger}-x_{\alpha,n,i}^{\delta,\ell}\|_{\mathcal{X}}&\leq& f(n)+
  O(\delta^{\frac{2i}{2i+1}})\qquad\mbox{as }\delta\rightarrow 0.
\end{eqnarray*}
\end{coro}

As mentioned at the end of Section~\ref{sec:iAT}, if an estimate of $\|x_n^{\dagger}\|_2$ 
is not available, then we may substitute $E$ by the expression 
$D\|x_{\alpha,n,i}^{\delta,\ell}\|_2$ with a constant $D\geq 1$. With this choice, for 
$\alpha$ satisfying \eqref{condECgk}, we achieve the same convergence rates. 

We note the improvement of the convergence rate $O(\delta^{2/3})$ for standard 
(non-iterated) Tikhonov regularization.


\section{An alternative parameter choice strategy}\label{impsel}
In this section we apply a parameter choice strategy that differs from the one above. Our aim is to be able to determine more accurate approximations of the desired solution $x^\dagger$ of \eqref{eq} and to compute these solutions in lower-dimensional 
Krylov subspaces.

We now apply the result of Appendix~\ref{ssec:new}. This requires that we translate the additional hypothesis of Proposition~\ref{comp} to establish convergence rates.

\begin{ass}\label{addass}
There hold the equalities
\begin{equation*}
    \tilde{\mathcal{R}}_{\ell}T_nx_n^{\dagger}=\tilde{T}_n^{(\ell)}x_n^{\dagger}
\end{equation*}
and
\begin{equation*}
    \mathcal{R}_{\ell}T_nx_n^{\dagger}=T_n^{(\ell)}x_n^{\dagger}.
\end{equation*}
\end{ass}

Note that, if the Arnoldi or Golub-Kahan processes do not break down, the operators 
$V_{n,\ell}V_{n,\ell}^{\ast}$ and $V_{\ell}V_{\ell}^{\ast}$ converges to $I_n$ as $\ell$ 
increases. Therefore, the hypotheses of Assumptions~\ref{addass} are close to being 
satisfied for $\ell$ large enough. We can estimate the norm of the differences 
\begin{equation*}
    \|(\tilde{\mathcal{R}}_{\ell}T_n-\tilde{T}_n^{(\ell)})x_n^{\dagger}\|_2=\|\tilde{\mathcal{R}}_{\ell}T_n(I_n-V_{n,\ell}V_{n,\ell}^{\ast})x_n^{\dagger}\|_2
\end{equation*}
and
\begin{equation*}
    \|(\mathcal{R}_{\ell}T_n-T_n^{(\ell)})x_n^{\dagger}\|_2=\|\mathcal{R}_{\ell}T_n(I_n-V_{\ell}V_{\ell}^{\ast})x_n^{\dagger}\|_2.
\end{equation*}
However, these quantities might not be monotonic functions of $\ell$; see 
Example~\ref{ex1} in Section~\ref{sec:comp}.

\begin{prop}\label{impconv}
Let Assumption~\ref{addass} hold and let $\tau=1$. Set $\tilde{h}_{\ell}=h_{\ell}=0$ and 
let \eqref{condECa} and \eqref{condECgk} be satisfied. Let $\tilde{\alpha}$ and $\alpha$ 
be the unique solutions of \eqref{conda} and \eqref{condgk}, respectively. Then for all 
$\alpha^{\prime}\geq\tilde{\alpha}$ and $\alpha^{\prime\prime}\geq\alpha$, we have that 
$\|x_n^{\dagger}-\tilde{x}_{\tilde{\alpha},n,i}^{\delta,\ell}\|_2\leq\|x_n^{\dagger}-
\tilde{x}_{\alpha^{\prime},n,i}^{\delta,\ell}\|_2$ and $\|x_n^{\dagger}-x_{\alpha,n,i}^{\delta,\ell}\|_2\leq\|x_n^{\dagger}-
x_{\alpha^{\prime\prime},n,i}^{\delta,\ell}\|_2$.
\end{prop}

\begin{proof}
The result follows from Proposition~\ref{comp} by using Assumption~\ref{addass}. 
\end{proof}

Under Assumption~\ref{addass}, the parameter choice strategy suggested by 
Proposition~\ref{impconv} results in computed solutions of higher quality than the 
parameter choice strategy suggested by Proposition~\ref{comp_iAT} and 
Proposition~\ref{comp_iGKT}. In fact, since for both equations \eqref{conda} and 
\eqref{condgk}, the left-hand sides are monotonically increasing, using smaller 
right-hand sides yields smaller values of the regularization parameter.
Moreover, since we now set $\tilde{h}_{\ell}=0$ and $h_{\ell}=0$, both equations 
\eqref{condECa} and \eqref{condECgk} are usually satisfied for smaller values of $\ell$.
This means that we can use Krylov spaces of smaller dimension.

\begin{rmk}
We may use the parameter choice strategy of this section also when Assumption~\ref{addass} 
is not satisfied. This is illustrated in the following section.
\end{rmk}


\section{Computed examples}\label{sec:comp}
We apply the iGKT regularization method to solve several ill-posed operator equations. 
Examples~\ref{ex4}, \ref{ex3} and \ref{ex1} are image deblurring problems. 
Example~\ref{ex1} compares the iGKT and iAT methods using the parameter choice strategy of
Section~\ref{impsel}. Example~\ref{ex5} works with rectangular matrices for which only the
iGKT method can be applied. All computations were carried out using MATLAB with about 
$15$ significant decimal digits.

The matrix $T_n$ takes on one of two forms: It either serves as a model for a 
blurring operator or models a computer tomography operator. The vector $x^{\dagger}_n\in\R^{n_2}$ is a discretization of the exact solution of \eqref{eq}. Its image $y_n=T_nx_n^{\dagger}$ is 
presumed impractical to measure directly. Instead, we know an observable, 
noise-contaminated vector, $y_n^{\delta}\in\R^{n_1}$, that is obtained by adding a vector that models noise to 
$y_n$. Let the vector $e_n\in\R^{n_1}$ have 
normally distributed random entries with zero mean. We scale this vector
\begin{equation*}
  \hat{e}_n \coloneqq \frac{\xi\|y_n\|_2}{\|e_n\|_2} \, e_n
\end{equation*}
to ensure a prescribed noise level $\xi>0$. Then we define 
\begin{equation*}
    y_n^\delta \coloneqq y_n + \hat{e}_n.
\end{equation*}
Clearly,
$\delta \coloneqq \| y_n^\delta - y_n\|_2 = \xi\|y_n\|_2$. We fix the value $\xi$ for each
example such that $\delta$ will correspond to $(100\cdot\xi)\%$ of $\|y\|_2$. To achieve 
replicability in the numerical examples, we define the ``noise'' deterministically by 
setting \texttt{seed}=11 in the MATLAB function \texttt{randn}, which generates normally 
distributed pseudorandom numbers and which we use to determine the entries of the vector $e_n$.

The low-rank approximation $T_n^{(\ell)}$ of $T_n$ is computed by applying $\ell$ steps of
the Golub-Kahan process to the matrix $T_n$ with initial vector 
$v_1=y_n^{\delta}/\|y_n^{\delta}\|_2$. Thus, we first evaluate the Golub-Kahan 
decomposition \eqref{gk} and then define the matrix $T_n^{(\ell)}$ by~\eqref{T_n-gk}. We
proceed similarly for the low-rank approximation $\tilde{T}_n^{(\ell)}$ defined by 
\eqref{T_n-a}. Note that these matrices are not explicitly formed.

We determine the parameter $\alpha$ for Algorithm \ref{algo:iGKT} (iGKT) by solving 
equation \eqref{condgk} with $C=1$ and $E=\|x_n^{\dagger}\|_2$, as suggested by 
Proposition~\ref{comp_iGKT}. Inequality \eqref{condECgk} holds for all examples of this 
section. In other words, $\alpha$ is the unique solution of
\begin{equation}\label{alpha_iGKT}
\alpha^{2i+1}(\hat{y}_{\ell+1}^\delta)^{\ast}(\Sigma_{\ell+1,\ell}
\Sigma_{\ell+1,\ell}^{\ast}+\alpha I_{\ell+1})^{-2i-1}\hat{y}_{\ell+1}^\delta=
(\|x_n^{\dagger}\|_2h_{\ell}+\delta)^2.
\end{equation}

We also illustrate the performance of the alternative parameter choice strategy of 
Section~\ref{impsel}. As suggested by Proposition~\ref{impconv}, the parameter $\alpha$ 
for Algorithm \ref{algo:iGKT} (iGKT) is the unique solution of
\begin{equation}\label{alpha_iGKTimp}
\alpha^{2i+1}(\hat{y}_{\ell+1}^\delta)^{\ast}(\Sigma_{\ell+1,\ell}
\Sigma_{\ell+1,\ell}^{\ast}+\alpha I_{\ell+1})^{-2i-1}\hat{y}_{\ell+1}^\delta=\delta^2.
\end{equation}
Since we would like to compare the iGKT and iAT methods using the alternative parameter 
choice strategies suggested by Proposition~\ref{impconv}, the parameter $\tilde{\alpha}$ 
for Algorithm~\ref{algo:iAT} (iAT) is chosen to be the unique solution of 
\begin{equation}\label{alpha_iATimp}
\tilde{\alpha}^{2i+1}(\tilde{\hat{y}}_{\ell+1}^\delta)^{\ast}(\tilde{\Sigma}_{\ell+1,\ell}
\tilde{\Sigma}_{\ell+1,\ell}^{\ast}+\tilde{\alpha} I_{\ell+1})^{-2i-1}\tilde{\hat{y}}_{\ell+1}^\delta=\delta^2.
\end{equation}


\begin{exmp}\label{ex4}
We consider an image deblurring problem and use the function \texttt{PRblurspeckle} from 
the toolbox \texttt{IRtools} \cite{IRTools} to determine an $n^2\times n^2$ matrix $T_n$ 
that models blurring of an image that is represented by $n\times n$ pixels. This function 
generates a blurred image with the blur determined by a speckle point spread function that 
models atmospheric blur.

We set $n=256$. The true image is represented by the vector $x_{n}^{\dagger}\in\R^{n^2}$. 
This image is shown in Figure~\ref{fig2D4} (upper left) and the observed image 
$y_{n}^{\delta}\in\R^{n^2}$ (upper right) is obtained by adding white Gaussian noise with 
$\xi=1\%$. We test the iGKT method for several values of $\ell$ and $i$. Table~\ref{tab4} 
reports the relative approximation error for the parameter $\alpha$ obtained by solving 
equations \eqref{alpha_iGKT}. The iGKT method can be seen to perform well.

\begin{table}
\caption{Example~\ref{ex4} - Relative error in approximate solutions computed by the iGKT 
method with parameter $\alpha$ determined using \eqref{alpha_iGKT} for $n=256$ and $\xi=1\%$.}\label{tab4}
\small
\begin{tabular}{cccc}
 \toprule%
 & & \multicolumn{2}{c}{iGKT}\\
 \cmidrule{3-4}
 $\ell$ & $i$ & $\alpha$ & $\|x_n^{\dagger}-x_{\alpha,n,i}^{\delta,\ell}\|_2/
 \|x_n^{\dagger}\|_2$\\
 \midrule
 \multirow{7}*{80} & 1 & $1.59\cdot 10^{1}$ & $9.76\cdot 10^{-1}$\\
 & 50 & $5.50\cdot 10^2$ & $9.65\cdot 10^{-1}$\\
 & 100 & $3.36\cdot 10^1$ & $6.44\cdot 10^{-1}$\\
 & 200 & $5.82\cdot 10^{0}$ & $4.73\cdot 10^{-1}$\\
 & 500 & $2.03\cdot 10^{0}$ & $3.24\cdot 10^{-1}$\\
 & 1000 & $1.42\cdot 10^{0}$ & $2.50\cdot 10^{-1}$\\
 & 2000 & $1.19\cdot 10^{0}$ & $2.17\cdot 10^{-1}$\\
 \midrule
 \multirow{7}*{160} & 1 & $1.03\cdot 10^{0}$ & $7.96\cdot 10^{-1}$\\
 & 50 & $4.56\cdot 10^{1}$ & $7.49\cdot 10^{-1}$\\
 & 100 & $3.36\cdot 10^1$ & $6.44\cdot 10^{-1}$\\
 & 200 & $5.82\cdot 10^{0}$ & $4.73\cdot 10^{-1}$\\
 & 500 & $2.03\cdot 10^{0}$ & $3.24\cdot 10^{-1}$\\
 & 1000 & $1.42\cdot 10^{0}$ & $2.50\cdot 10^{-1}$\\
 & 2000 & $1.19\cdot 10^{0}$ & $2.17\cdot 10^{-1}$\\
 \midrule
 \multirow{7}*{240} & 1 & $4.55\cdot 10^{-1}$ & $7.11\cdot 10^{-1}$\\
 & 50 & $2.43\cdot 10^{1\phantom{-}}$ & $6.77\cdot 10^{-1}$\\
 & 100 & $3.36\cdot 10^{1\phantom{-}}$ & $6.44\cdot 10^{-1}$\\
 & 200 & $5.82\cdot 10^{0\phantom{-}}$ & $4.73\cdot 10^{-1}$\\
 & 500 & $2.03\cdot 10^{0\phantom{-}}$ & $3.24\cdot 10^{-1}$\\
 & 1000 & $1.42\cdot 10^{0\phantom{-}}$ & $2.50\cdot 10^{-1}$\\
 & 2000 & $1.19\cdot 10^{0\phantom{-}}$ & $2.17\cdot 10^{-1}$\\
 \bottomrule
\end{tabular}
\end{table}

Table~\ref{tab4imp} reports the relative approximation errors for the parameter $\alpha$ 
obtained by solving equation \eqref{alpha_iGKTimp}. Assumption~\ref{addass} is not satisfied 
for this example. Nevertheless, we can try to use this parameter choice rule because it
allows a smaller value of $\ell$ without loss in the quality of the reconstructions.
Figure~\ref{fig2D4} shows some reconstructed images in the second row.

\begin{table}
\caption{Example~\ref{ex4} - Relative error in approximate solutions computed by the iGKT method with parameter $\alpha$ determined using \eqref{alpha_iGKTimp} for $n=256$ and $\xi=1\%$.}\label{tab4imp}
\small
\begin{tabular}{cccc}
 \toprule%
 & & \multicolumn{2}{c}{iGKT}\\
 \cmidrule{3-4}
 $\ell$ & $i$ & $\alpha$ & $\|x_n^{\dagger}-x_{\alpha,n,i}^{\delta,\ell}\|_2/
 \|x_n^{\dagger}\|_2$\\
 \midrule
 \multirow{7}*{20} & 1 & $4.17\cdot 10^{-3}$ & $3.75\cdot 10^{-1}$\\
 & 50 & $2.58\cdot 10^{-1}$ & $3.54\cdot 10^{-1}$\\
 & 100 & $5.17\cdot 10^{-1}$ & $3.54\cdot 10^{-1}$\\
 & 200 & $1.04\cdot 10^{0\phantom{-}}$ & $3.54\cdot 10^{-1}$\\
 & 500 & $2.03\cdot 10^{0\phantom{-}}$ & $3.42\cdot 10^{-1}$\\
 & 1000 & $1.42\cdot 10^{0\phantom{-}}$ & $3.24\cdot 10^{-1}$\\
 & 2000 & $1.19\cdot 10^{0\phantom{-}}$ & $3.24\cdot 10^{-1}$\\
 \midrule
 \multirow{7}*{40} & 1 & $2.23\cdot 10^{-3}$ & $3.09\cdot 10^{-1}$\\
 & 50 & $1.19\cdot 10^{-1}$ & $2.84\cdot 10^{-1}$\\
 & 100 & $2.37\cdot 10^{-1}$ & $2.83\cdot 10^{-1}$\\
 & 200 & $4.75\cdot 10^{-1}$ & $2.83\cdot 10^{-1}$\\
 & 500 & $1.19\cdot 10^{0\phantom{-}}$ & $2.83\cdot 10^{-1}$\\
 & 1000 & $1.42\cdot 10^{0\phantom{-}}$ & $2.53\cdot 10^{-1}$\\
 & 2000 & $1.19\cdot 10^{0\phantom{-}}$ & $2.34\cdot 10^{-1}$\\
 \midrule
 \multirow{7}*{80} & 1 & $1.59\cdot 10^{-3}$ & $2.84\cdot 10^{-1}$\\
 & 50 & $8.32\cdot 10^{-2}$ & $2.59\cdot 10^{-1}$\\
 & 100 & $1.67\cdot 10^{-1}$ & $2.59\cdot 10^{-1}$\\
 & 200 & $3.33\cdot 10^{-1}$ & $2.59\cdot 10^{-1}$\\
 & 500 & $8.34\cdot 10^{-1}$ & $2.59\cdot 10^{-1}$\\
 & 1000 & $1.42\cdot 10^{0\phantom{-}}$ & $2.50\cdot 10^{-1}$\\
 & 2000 & $1.19\cdot 10^{0\phantom{-}}$ & $2.17\cdot 10^{-1}$\\
 \bottomrule
\end{tabular}
\end{table}

\begin{figure}
\centering
{\includegraphics[scale=1]{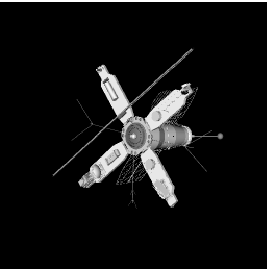}} 
\hspace{2cm}
{\includegraphics[scale=1]{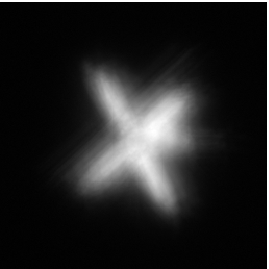}}\\
{\includegraphics[scale=1]{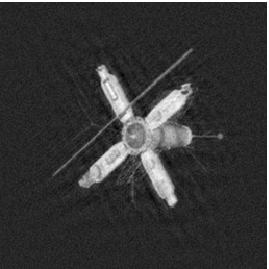}}
\hspace{2cm}
{\includegraphics[scale=1]{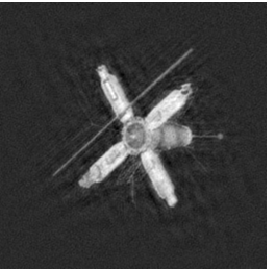}}
\caption{Example~\ref{ex4} - Exact solution $x^{\dagger}_{n}$ (upper left) and observed 
image $y_n^{\delta}$ (upper right). Approximate solutions $x^{\delta,\ell}_{\alpha,n,i}$ 
computed by the iGKT method with $\ell=80$ and $i=2000$ with $\alpha$ determined using 
\eqref{alpha_iGKT} (lower left) and $x^{\delta,\ell}_{\alpha,n,i}$ computed by the iGKT 
method with $\ell=40$ and $i=2000$ with $\alpha$ determined using \eqref{alpha_iGKTimp} 
(lower right). Here $n=256$ and $\xi= 1\%$.}\label{fig2D4}
\end{figure}

\end{exmp}


\begin{exmp}\label{ex3}
We consider another image deblurring problem and use the function \texttt{PRblurshake} 
from \texttt{IRtools} with option \texttt{BlurLevel}=\texttt{`severe'}. The MATLAB 
function \texttt{rng} is applied with the parameter values \texttt{seed}=11 and 
\texttt{generator}=\texttt{`twister'} to determine an $n^2\times n^2$ blurring matrix 
$T_{n}$ and a blurred and noisy image that is represented by $n\times n$ pixels. The 
blur simulates random camera motion (shaking). The noise is white Gaussian with $\xi=1\%$.

We set $n=256$; the true image is represented by the vector $x_{n}^{\dagger}\in\R^{n^2}$.
This image is shown in Figure~\ref{fig2D3} (upper left) with the observed image 
$y_{n}^{\delta}\in\R^{n^2}$ (upper right).

Table~\ref{tab3} shows results for the iGKT method with the relative error of the computed 
approximations corresponding to the parameter $\alpha$ that is determined by solving 
equation~\eqref{alpha_iGKT} or  equation \eqref{alpha_iGKTimp} for several values of 
$\ell$ and $i$. Assumptions~\ref{addass} are not satisfied. However, the alternative 
parameter selection strategy allows smaller values of $\ell$.

\begin{table}
\caption{Example~\ref{ex3} - Relative error in approximate solutions computed by the iGKT method with parameter $\alpha$ determined using \eqref{alpha_iGKT} and \eqref{alpha_iGKTimp}, for $n=256$ and $\xi= 1\%$.}\label{tab3}
\small
\begin{tabular}{ccccccc}
 \toprule%
 & & \multicolumn{2}{c}{iGKT \eqref{alpha_iGKT}} & & \multicolumn{2}{c}{iGKT \eqref{alpha_iGKTimp}} \\
 \cmidrule{3-4}
 \cmidrule{6-7}
 $\ell$ & $i$ & $\alpha$ & $\|x_n^{\dagger}-x_{\alpha,n,i}^{\delta,\ell}\|_2/
 \|x_n^{\dagger}\|_2$ &  & $\tilde{\alpha}$ & $\|x_n^{\dagger}-\tilde{x}_{\tilde{\alpha},n,i}^{\delta,\ell}\|_2/\|x_n^{\dagger}\|_2$\\
 \midrule
 \multirow{5}*{30} & 1 & - & - & & $7.56\cdot 10^{-3}$ & $1.80\cdot 10^{-1}$\\
 & 50 & - & - & & $4.04\cdot 10^{-1}$ & $1.64\cdot 10^{-1}$\\
 & 100 & - & - & & $8.08\cdot 10^{-1}$ & $1.64\cdot 10^{-1}$\\
 & 200 & - & - & & $1.62\cdot 10^{0\phantom{-}}$ & $1.64\cdot 10^{-1}$\\
 & 500 & - & - & & $2.02\cdot 10^{0\phantom{-}}$ & $1.45\cdot 10^{-1}$\\
 \midrule
 \multirow{5}*{60} & 1 & $1.81\cdot 10^{1}$ & $9.63\cdot 10^{-1}$ & & $6.17\cdot 10^{-3}$ & $1.72\cdot 10^{-1}$\\
 & 50 & $6.23\cdot 10^2$  & $9.46\cdot 10^{-1}$ & & $3.26\cdot 10^{-1}$ & $1.57\cdot 10^{-1}$\\
 & 100 & $3.35\cdot 10^1$  & $3.91\cdot 10^{-1}$ & & $6.53\cdot 10^{-1}$  & $1.57\cdot 10^{-1}$\\
 & 200 & $5.81\cdot 10^{0}$ & $2.25\cdot 10^{-1}$ & & $1.31\cdot 10^{0\phantom{-}}$  & $1.57\cdot 10^{-1}$\\
 & 500 & $2.02\cdot 10^{0}$ & $1.45\cdot 10^{-1}$ & & $2.02\cdot 10^{0\phantom{-}}$  & $1.45\cdot 10^{-1}$\\
 \midrule
 \multirow{5}*{120}  & 1 & $1.60\cdot 10^{0}$ & $7.27\cdot 10^{-1}$ & & $5.81\cdot 10^{-3}$ & $1.69\cdot 10^{-1}$\\
 & 50 & $6.71\cdot 10^1$ & $6.43\cdot 10^{-1}$ & & $3.07\cdot 10^{-1}$ & $1.55\cdot 10^{-1}$\\
 & 100 & $3.35\cdot 10^1$ & $3.91\cdot 10^{-1}$ & & $6.14\cdot 10^{-1}$ & $1.55\cdot 10^{-1}$\\
 & 200 & $5.81\cdot 10^{0}$ & $2.25\cdot 10^{-1}$ & & $1.23\cdot 10^{0\phantom{-}}$ & $1.55\cdot 10^{-1}$\\
 & 500 & $2.02\cdot 10^{0}$ & $1.45\cdot 10^{-1}$ & & $2.02\cdot 10^{0\phantom{-}}$ & $1.45\cdot 10^{-1}$\\
 \bottomrule
\end{tabular}
\end{table}

Reconstructed images for both parameter choice strategies are shown in Figure~\ref{fig2D3}.

\begin{figure}
\centering
{\includegraphics[scale=1]{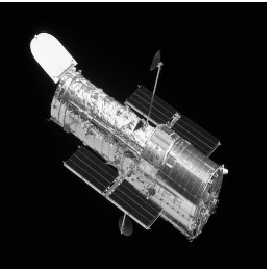}} 
\hspace{2cm}
{\includegraphics[scale=1]{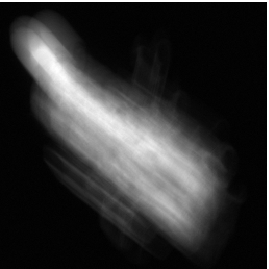}}\\
{\includegraphics[scale=1]{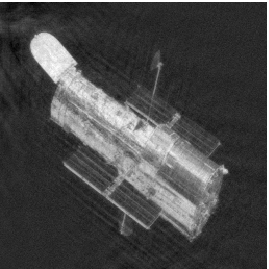}}
\hspace{2cm}
{\includegraphics[scale=1]{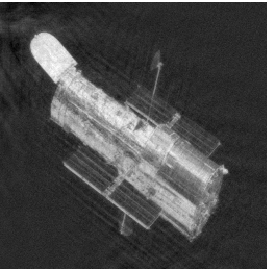}}
\caption{Example~\ref{ex3} - Exact solution $x^{\dagger}_{n}$ (Up-Left) and observed image
$y_n^{\delta}$ (upper right). Approximate solutions $x^{\delta,\ell}_{\alpha,n,i}$ 
computed by the iGKT method with $\ell=60$, $i=500$, and $\alpha$ determined by 
solving \eqref{alpha_iGKT} (lower left) and $x^{\delta,\ell}_{\alpha,n,i}$ computed by the 
iGKT method with $\ell=30$, $i=500$, and $\alpha$ determined using \eqref{alpha_iGKTimp} 
(lower right). Here $n=256$ and $\xi= 1\%$.}\label{fig2D3}
\end{figure}

\end{exmp}


\begin{exmp}\label{ex1}
This example considers the restoration of an image that has been contaminated by motion
blur and noise. Thus, we use an $n^2\times n^2$ psfMatrix $T_n$ that simulates motion blur
with $n=256$. Figure~\ref{fig2D1} shows the true image represented by 
$x_n^{\dagger}\in\R^{n^2}$ (upper left) and the observed blurred and noisy image 
represented by $y_n^{\delta}\in\R^{n^2}$ (upper right). The noise is white Gaussian with 
$\xi=2\%$.

We compare the iGKT and iAT methods for several values of $\ell$ and $i$. The inequalities 
\eqref{condECa} and \eqref{condECgk} are not satisfied. We show results for the alternative 
parameter choice strategy. Table~\ref{tab1imp} reports the relative approximation errors for 
the parameters $\alpha$ and  $\tilde\alpha$, obtained by solving equations 
\eqref{alpha_iGKTimp} and \eqref{alpha_iATimp}, respectively. Assumptions~\ref{addass} are 
not satisfied. The iGKT method gives better reconstructions than the iAT method. We also 
report the average computing times. They show the iAT method to be slightly faster.

\begin{table}
\caption{Example~\ref{ex1} - Relative error in approximate solutions computed by the iGKT
and iAT methods with the parameters $\alpha$ and $\tilde{\alpha}$ determined by using 
\eqref{alpha_iGKTimp} and \eqref{alpha_iATimp}, respectively. Average computing times are
reported and $n=256$ and $\xi=2\%$.}\label{tab1imp}
\footnotesize
\begin{tabular}{ccccccccc}
 \toprule%
 & & \multicolumn{3}{c}{iGKT} & & \multicolumn{3}{c}{iAT} \\
 \cmidrule{3-5}
 \cmidrule{7-9}
 $\ell$ & $i$ & $\alpha$ & $\|x_n^{\dagger}-x_{\alpha,n,i}^{\delta,\ell}\|_2/
 \|x_n^{\dagger}\|_2$ & \text{time (s)} &  & $\tilde{\alpha}$ & $\|x_n^{\dagger}-\tilde{x}_{\tilde{\alpha},n,i}^{\delta,\ell}\|_2/\|x_n^{\dagger}\|_2$ & \text{time (s)}\\
 \midrule
 \multirow{5}*{10} & 1 & $6.10\cdot 10^{-2}$ & $1.35\cdot 10^{-1}$ & $0.26$ & & $6.25\cdot 10^{-2}$ & $1.49\cdot 10^{-1}$ & $0.16$\\
 & 50 & $4.83\cdot 10^{0\phantom{-}}$ & $1.24\cdot 10^{-1}$ & $0.26$ & & $5.20\cdot 10^{0\phantom{-}}$ & $1.43\cdot 10^{-1}$ & $0.16$\\
 & 100 & $9.66\cdot 10^{0\phantom{-}}$ & $1.24\cdot 10^{-1}$ & $0.27$ & & $1.04\cdot 10^{1\phantom{-}}$ & $1.43\cdot 10^{-1}$ & $0.16$\\
 & 150 & $1.02\cdot 10^{0\phantom{-}}$ & $1.15\cdot 10^{-1}$ & $0.27$ & & $1.02\cdot 10^{1\phantom{-}}$ & $1.46\cdot 10^{-1}$ & $0.17$\\
 & 200 & $5.72\cdot 10^{0\phantom{-}}$ & $1.02\cdot 10^{-1}$ & $0.27$ & & $5.72\cdot 10^{0\phantom{-}}$ & $1.57\cdot 10^{-1}$ & $0.17$\\
 \midrule
 \multirow{5}*{20} & 1 & $4.96\cdot 10^{-2}$ & $1.26\cdot 10^{-1}$ & $0.53$ & & $5.17\cdot 10^{-2}$ & $1.47\cdot 10^{-1}$ & $0.32$\\
 & 50 & $3.30\cdot 10^{0\phantom{-}}$ & $1.14\cdot 10^{-1}$ & $0.53$ & & $3.69\cdot 10^{0\phantom{-}}$ & $1.44\cdot 10^{-1}$ & $0.32$\\
 & 100 & $6.60\cdot 10^{0\phantom{-}}$ & $1.14\cdot 10^{-1}$ & $0.54$ & & $7.38\cdot 10^{0\phantom{-}}$ & $1.44\cdot 10^{-1}$ & $0.31$\\
 & 150 & $9.90\cdot 10^{0\phantom{-}}$ & $1.14\cdot 10^{-1}$ & $0.54$ & & $1.02\cdot 10^{1\phantom{-}}$ & $1.45\cdot 10^{-1}$ & $0.32$\\
 & 200 & $5.72\cdot 10^{0\phantom{-}}$ & $1.00\cdot 10^{-1}$ & $0.54$ & & $5.72\cdot 10^{0\phantom{-}}$ & $1.58\cdot 10^{-1}$ & $0.32$\\
 \midrule
 \multirow{5}*{30} & 1 & $4.71\cdot 10^{-2}$ & $1.24\cdot 10^{-1}$ & $0.80$ & & $5.12\cdot 10^{-2}$ & $1.46\cdot 10^{-1}$ & $0.48$\\
 & 50 & $3.03\cdot 10^{0\phantom{-}}$ & $1.12\cdot 10^{-1}$ & $0.80$ & & $3.63\cdot 10^{0\phantom{-}}$ & $1.44\cdot 10^{-1}$ & $0.47$\\
 & 100 & $6.06\cdot 10^{0\phantom{-}}$ & $1.12\cdot 10^{-1}$ & $0.80$ & & $7.25\cdot 10^{0\phantom{-}}$ & $1.44\cdot 10^{-1}$ & $0.48$\\
 & 150 & $9.09\cdot 10^{0\phantom{-}}$ & $1.12\cdot 10^{-1}$ & $0.81$ & & $1.02\cdot 10^{0\phantom{-}}$ & $1.45\cdot 10^{-1}$ & $0.48$\\
 & 200 & $5.72\cdot 10^{0\phantom{-}}$ & $1.00\cdot 10^{-1}$ & $0.80$ & & $5.72\cdot 10^{0\phantom{-}}$ & $1.58\cdot 10^{-1}$ & $0.48$\\
 \bottomrule
\end{tabular}
\end{table}

Figure~\ref{figass} displays the relative norm of the difference from 
Assumption~\ref{addass} for several values of $\ell$. We can see that for the iGKT method 
this norm decrease rapidly and monotonically, but not for the iAT method. The figure also
shows the values of $h_{\ell}$ and $\tilde{h}_{\ell}$.

\begin{figure}
\centering
{\includegraphics[scale=0.3]{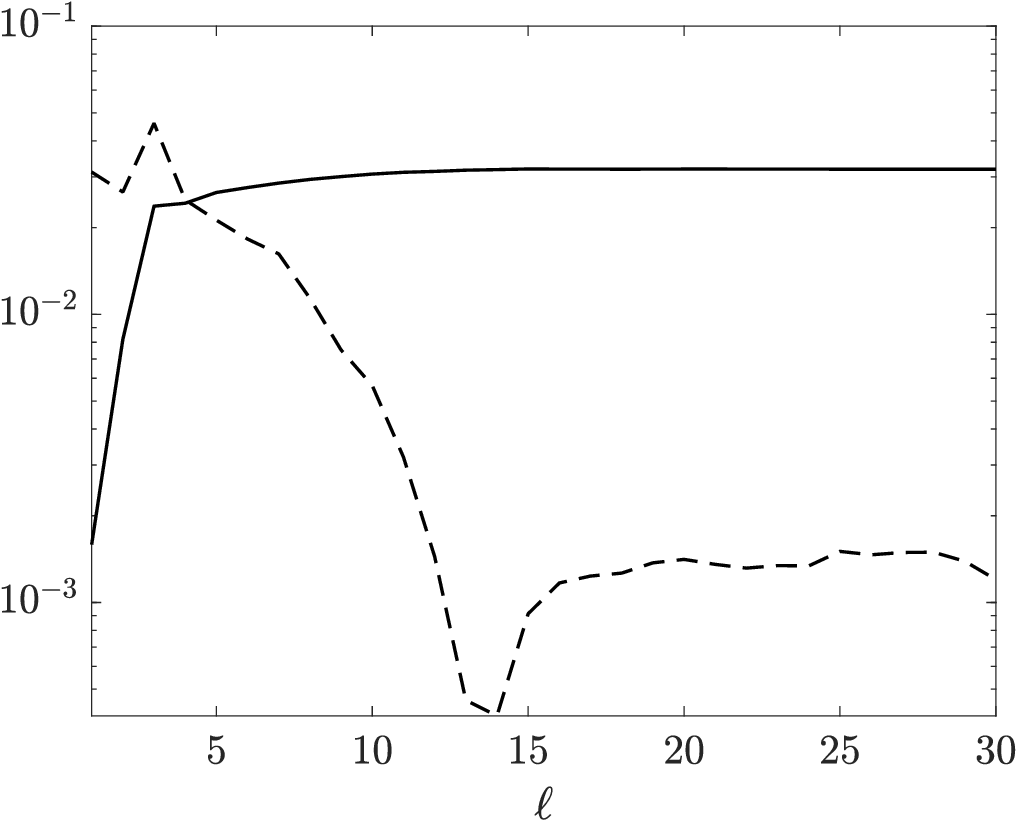}} 
\hspace{2cm}
{\includegraphics[scale=0.3]{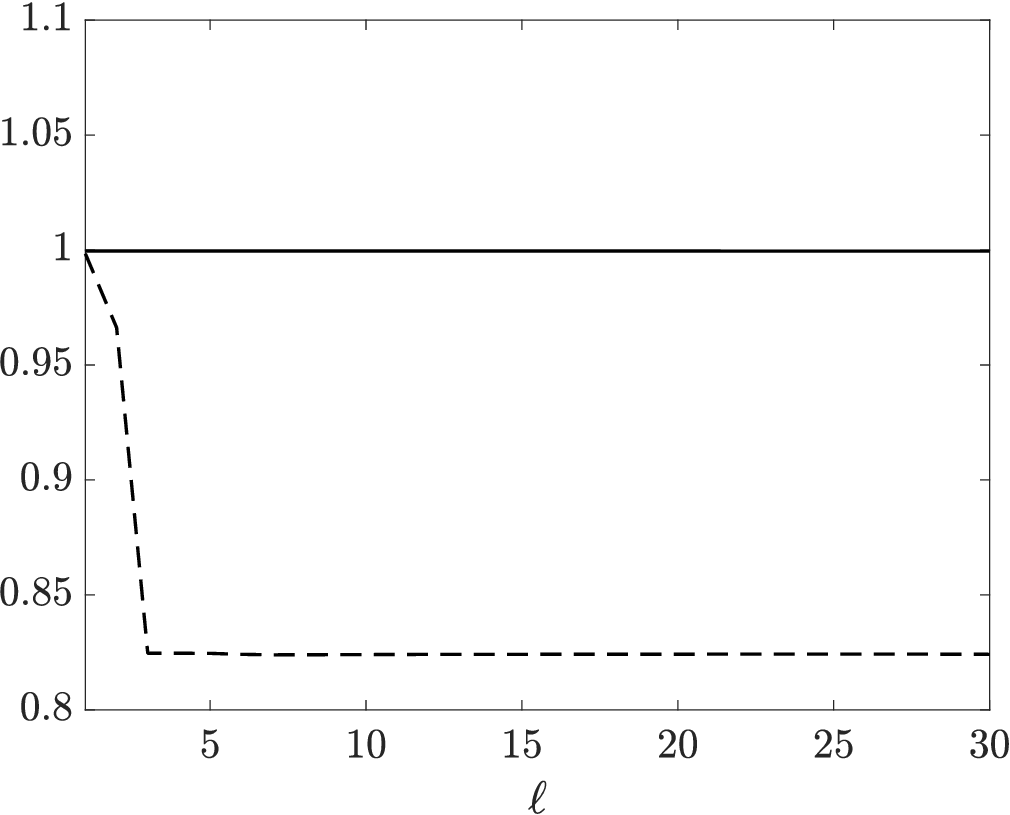}} 
\caption{Example~\ref{ex1} - (Left) Values of $\frac{\|(\tilde{\mathcal{R}}_{\ell}T_n-\tilde{T}_n^{(\ell)})x_n^{\dagger}\|_2}{\|\tilde{\mathcal{R}}_{\ell}T_nx_n^{\dagger}\|_2}$ (continuous line) and $\frac{\|(\mathcal{R}_{\ell}T_n-T_n^{(\ell)})x_n^{\dagger}\|_2}{\|\mathcal{R}_{\ell}T_nx_n^{\dagger}\|_2}$ (dashed line) in logarithmic scale. (Right) Values of $\frac{\tilde{h}_{\ell}}{\|T_n\|_2}$ (continuous line) and $\frac{h_{\ell}}{\|T_n\|_2}$ (dashed line). Here $\ell=1,\ldots,30$.}\label{figass}
\end{figure}

Figure~\ref{fig2D1} shows some reconstructed images in the second row.

\begin{figure}
\centering
{\includegraphics[scale=1]{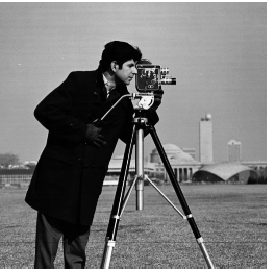}} 
\hspace{2cm}
{\includegraphics[scale=1]{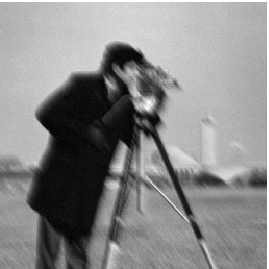}}\\
{\includegraphics[scale=1]{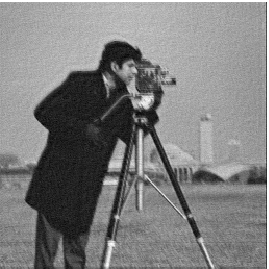}}
\hspace{2cm}
{\includegraphics[scale=1]{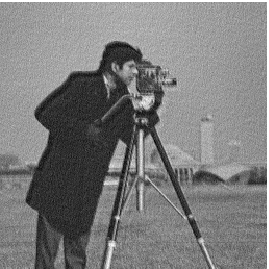}}
\caption{Example~\ref{ex1} - True image $x^{\dagger}_{n}$ (upper left) and observed image 
$y_n^{\delta}$ (upper right). Approximate solutions $x^{\delta,\ell}_{\alpha,n,i}$ computed
by the iGKT method with $i=200$ and $\alpha$ determined using \eqref{alpha_iGKTimp} 
(lower left) and $x^{\delta,\ell}_{\alpha,n,i}$ computed by the iAT method with $i=100$ and
$\tilde{\alpha}$ determined using \eqref{alpha_iATimp} (lower right) with $\ell=20$. 
Here $n=256$ and $\xi= 2\%$.}\label{fig2D1}
\end{figure}
\end{exmp}


\begin{exmp}\label{ex5}
This example is concerned with a computerized tomography problem. We use the function 
\texttt{PRtomo} from \texttt{IRtools} setting $\lfloor n\sqrt{2}\rfloor$ number of rays
and $180$ uniformly distributed angles from $0$ to $\pi$, to determine an 
$180\lfloor n\sqrt{2}\rfloor\times n^2$ matrix $T_{n}$. The true image, represented by the
vector $x_{n}^{\dagger}\in\R^{n^2}$ for $n=256$, is shown in Figure~\ref{fig2D5} 
(upper left). The observed data (the sinogram) 
$y_{n}^{\delta}\in\R^{180\lfloor n\sqrt{2}\rfloor}$ is depicted in Figure~\ref{fig2D5} 
(upper right). The noise is $\xi=1\%$.

Table~\ref{tab5} collects some relative approximation errors for the iGKT method. The 
parameter $\alpha$ is obtained by solving equation \eqref{alpha_iGKT} or by solving 
equation \eqref{alpha_iGKTimp} for various values of $\ell$ and $i$. Assumption~\ref{addass} 
is not satisfied. The alternative parameter selection strategy allows smaller values of 
$\ell$ and yields good reconstructions for small values of $i$.

\begin{table}
\caption{Example~\ref{ex5} - Relative error in approximate solutions computed by the iGKT method with parameter $\alpha$ determined using \eqref{alpha_iGKT} and \eqref{alpha_iGKTimp}, for $n=256$ and $\xi= 1\%$.}\label{tab5}
\small
\begin{tabular}{ccccccc}
 \toprule%
 & & \multicolumn{2}{c}{iGKT \eqref{alpha_iGKT}} & & \multicolumn{2}{c}{iGKT \eqref{alpha_iGKTimp}} \\
 \cmidrule{3-4}
 \cmidrule{6-7}
 $\ell$ & $i$ & $\alpha$ & $\|x_n^{\dagger}-x_{\alpha,n,i}^{\delta,\ell}\|_2/
 \|x_n^{\dagger}\|_2$ &  & $\tilde{\alpha}$ & $\|x_n^{\dagger}-\tilde{x}_{\tilde{\alpha},n,i}^{\delta,\ell}\|_2/\|x_n^{\dagger}\|_2$\\
 \midrule
 \multirow{4}*{3} & 1 & - & - & & $6.36\cdot 10^{2}$ & $5.65\cdot 10^{-1}$\\
 & 25 & - & - & & $3.89\cdot 10^{4}$ & $5.53\cdot 10^{-1}$\\
 & 50 & - & - & & $1.03\cdot 10^{3}$ & $5.47\cdot 10^{-1}$\\
 & 100 & - & - & & $3.27\cdot 10^{1}$ & $5.47\cdot 10^{-1}$\\
 \midrule
 \multirow{4}*{6} & 1 & $1.40\cdot 10^{6}$ & $9.90\cdot 10^{-1}$ & & $2.34\cdot 10^{2}$ & $3.34\cdot 10^{-1}$\\
 & 25 & $9.29\cdot 10^{5}$  & $8.06\cdot 10^{-1}$ & & $1.13\cdot 10^{4}$ & $3.12\cdot 10^{-1}$\\
 & 50 & $1.03\cdot 10^{3}$  & $2.99\cdot 10^{-1}$ & & $1.03\cdot 10^{3}$ & $2.99\cdot 10^{-1}$\\
 & 100 & $3.27\cdot 10^{1}$  & $2.99\cdot 10^{-1}$ & & $3.26\cdot 10^{1}$ & $2.99\cdot 10^{-1}$\\
 \midrule
 \multirow{4}*{12} & 1 & $2.59\cdot 10^{5}$ & $9.52\cdot 10^{-1}$ & & $1.47\cdot 10^{2}$ & $2.22\cdot 10^{-1}$\\
 & 25 & $9.29\cdot 10^{5}$  & $8.06\cdot 10^{-1}$ & & $5.76\cdot 10^{3}$ & $1.96\cdot 10^{-1}$\\
 & 50 & $1.03\cdot 10^{3}$  & $1.78\cdot 10^{-1}$ & & $1.03\cdot 10^{3}$ & $1.78\cdot 10^{-1}$\\
 & 100 & $3.27\cdot 10^{1}$  & $1.78\cdot 10^{-1}$ & & $3.27\cdot 10^{1}$ & $1.78\cdot 10^{-1}$\\
 \bottomrule
\end{tabular}
\end{table}

Two reconstructions are shown in the second row of Figure~\ref{fig2D5}.

\begin{figure}
\centering
{\includegraphics[scale=1]{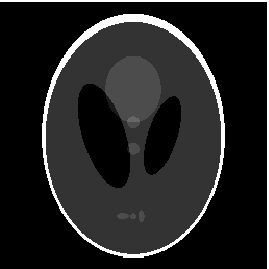}} 
\hspace{2cm}
{\includegraphics[width=4.6cm, height=4.6cm]{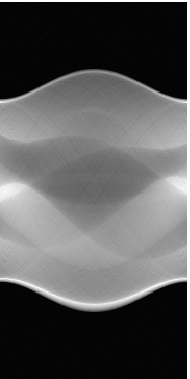}}\\
{\includegraphics[scale=1]{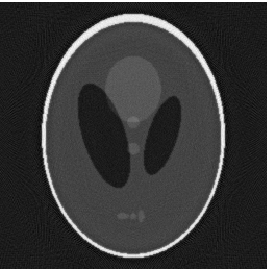}}
\hspace{2cm}
{\includegraphics[scale=1]{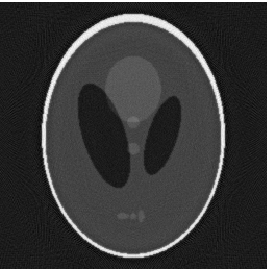}}
\caption{Example~\ref{ex5} - Exact solution $x^{\dagger}_{n}$ (Up-Left) and observed data 
(sinogram) $y_n^{\delta}$ (upper right). Approximate solutions $x^{\delta,\ell}_{\alpha,n,i}$ 
computed by the iGKT method with $\alpha$ determined by \eqref{alpha_iGKT} (lower left) 
and $x^{\delta,\ell}_{\alpha,n,i}$ computed by the iGKT method with $\alpha$ determined by
\eqref{alpha_iGKTimp} (lower right) with $\ell=12$ and $i=50$. Here $n=256$ and 
$\xi= 1\%$.}\label{fig2D5}
\end{figure}
\end{exmp}

\section{Conclusion}\label{sec:end}
The paper presents a convergence analysis for iterated Tikhonov regularization based on
partial Golub-Kahan bidiagonalization. Two approaches for choosing the regularization 
parameter are discussed. Our convergence analysis improves the convergence analysis in
\cite{bianchi2023itat}.

\appendix
\setcounter{section}{0}  

\section{Appendix}\label{sec:appA}

\renewcommand{\theequation}{\thesection.\arabic{equation}} 
\makeatletter  
\@addtoreset{equation}{section} 
\makeatother

\newtheorem{appendixtheorem}{Theorem}[section]  
\newtheorem{appendixproposition}[appendixtheorem]{Proposition}
\newtheorem{appendixlemma}[appendixtheorem]{Lemma}
\renewcommand{\theappendixtheorem}{\thesection.\arabic{appendixtheorem}}
\renewcommand{\theappendixproposition}{\thesection.\arabic{appendixtheorem}}
\renewcommand{\theappendixlemma}{\thesection.\arabic{appendixtheorem}}

In this section, we provide the proofs for our results.
The first subsection provides the analysis applied in Section~\ref{impsel} and
the second subsection gives more details of some of the proofs in \cite{bianchi2023itat} that may be of interest to a reader. Moreover, we present an updated version of these results by modifying a specific condition.

Our analysis is done in a general setting. To enhance readability, we provide Table~\ref{tabnot}, which connects the notation used in this 
section with the notation used for the iAT and iGKT methods in the previous sections.

\begin{table}[h!]
    \centering
    \begin{tabular}{|c|c|c|c|c|c|c|}
        \hline
        Notation of this section & \( T \) & \( T_h \) & \(  Q_m  \) & \(  T_{h,m}\coloneqq Q_mT_h \) & \(  x^{\dagger} \) & \(  x_{\alpha,m,i}^{\delta,h} \) \\
        \hline
        Notation of iAT & \( T_n \) & \( \tilde{T}_n^{(\ell)} \) & \( \tilde{\mathcal{R}}_{\ell} \) & \( T_n^{(\ell)} \) & \(  x_n^{\dagger} \) & \(  \tilde{x}_{\tilde{\alpha},n,i}^{\delta,\ell} \) \\
        \hline
        Notation of iGKT & \( T_n \) & \( T_n^{(\ell)} \) & \( \mathcal{R}_{\ell} \) & \( T_n^{(\ell)} \) & \(  x_n^{\dagger} \) & \(  x_{\alpha,n,i}^{\delta,\ell} \) \\
        \hline
    \end{tabular}
    \caption{Comparison of notation.}\label{tabnot}
\end{table}

\subsection{A different parameter choice strategy}\label{ssec:new}

\begin{ass}\label{assumptions}
Let $T\colon\mathcal{X}\to \mathcal{Y}$ denote a bounded linear operator between separable
Hilbert spaces $\mathcal{X}$ and $\mathcal{Y}$ with norms $\|\cdot\|_{\mathcal X}$ and
$\|\cdot\|_{\mathcal Y}$, respectively, that are induced by inner products, and let $T_h$ 
be an approximation of $T$. Consider a family $\lbrace W_m\rbrace_{m\in\N}$ of 
finite-dimensional subspaces of $\mathcal{Y}$ such that the orthogonal projector $Q_m$ 
into $W_m$ converges to $I$ on $\overline{\Range{T}}$.
\end{ass}

Define the operator $T_{h,m}\coloneqq Q_mT_h$. Using the iterated Tikhonov (iT) method 
(see \cite[Section 5]{engl1996}) applied to the solution of equation \eqref{eqdelta} with 
the operator $T_{h,m}$, we define the approximate solution
\begin{equation}\label{it}
x_{\alpha,m,i}^{\delta,h}\coloneqq\sum_{k=1}^i\alpha^{k-1}(T_{h,m}^{\ast}T_{h,m}+
\alpha I)^{-k}T_{h,m}^{\ast}y^{\delta}
\end{equation}
of \eqref{eqdelta}. Consider for $\tau\geq 1$ the equation for $\alpha$,
\begin{equation}\label{cond2}
\alpha^{2i+1}\langle (T_{h,m}T_{h,m}^{\ast}+\alpha I)^{-2i-1}Q_my^{\delta},Q_my^{\delta}
\rangle=\tau\delta^2.
\end{equation}
Similarly as in \cite[Proposition A.6]{bianchi2023itat}, one can show that there is a 
unique solution $\alpha$ of \eqref{cond2} provided that
\begin{equation}\label{condtau}
  \tau\delta^2<\|Q_my^{\delta}\|^2.
\end{equation}

\begin{appendixproposition}\label{comp}
Let Assumption~\ref{assumptions} and \eqref{condtau} hold with $\tau=1$. If 
$Q_mTx^{\dagger}=T_{h,m}x^{\dagger}$ and $\alpha$ is the unique solution of \eqref{cond2},
then for all $\alpha^{\prime}\geq\alpha$, we have 
$\|x^{\dagger}-x_{\alpha,m,i}^{\delta,h}\|\leq
\|x^{\dagger}-x_{\alpha^{\prime},m,i}^{\delta,h}\|$.
\end{appendixproposition}

\begin{proof}
Let $\lbrace F_{\mu}^{h,m}\rbrace_{\mu\in\R}$ be a spectral family for 
$T_{h,m}T^{\ast}_{h,m}$. We recall that for any self-adjoint operator $A$, there exists a 
unique spectral family $\lbrace E_{\mu}\rbrace_{\mu\in\R}$ that is a collection of 
projectors such that 
\begin{equation*}
f(A)x=\int_{0}^{\infty}f(\mu)dE_{\mu}x\coloneqq\lim_{n\to\infty}
\sum_{k=1}^nf(\mu_k)(E_{\mu_{k}}-E_{\mu_{k-1}})x,
\end{equation*}
where on the right-hand side we are considering the limit of a Riemann sum for any 
continuous function $f$ and $x\in\mathcal{X}$; see \cite[Section 2.3]{engl1996} for a 
precise definition.

Define $e(\alpha)\coloneqq\frac{1}{2}\|x^{\dagger}-x_{\alpha,m,i}^{\delta,h}\|^2$. It 
follows from the hypothesis that $T_{h,m}x^{\dagger}=Q_mTx^{\dagger}=Q_my$ and
\begin{align*}
\frac{de(\alpha)}{d\alpha}&=i\biggl\langle T_{h,m}x^{\dagger}-\int_0^{\infty}\frac{(\mu+\alpha)^i-\alpha^i}{(\mu+\alpha)^i}dF_{\mu}^{h,m}Q_my^{\delta},\int_{0}^{\infty}\frac{\alpha^{i-1}}{(\mu+\alpha)^{i+1}}dF_{\mu}^{h,m}Q_my^{\delta}\biggr\rangle\\
&=i\biggl\langle Q_my-\int_0^{\infty}\frac{(\mu+\alpha)^i-\alpha^i}{(\mu+\alpha)^i}dF_{\mu}^{h,m}Q_my^{\delta},\int_{0}^{\infty}\frac{\alpha^{i-1}}{(\mu+\alpha)^{i+1}}dF_{\mu}^{h,m}Q_my^{\delta}\biggr\rangle.
\end{align*}
Adding and subtracting 
$i\alpha^{2i-1}\|(T_{h,m}T_{h,m}^{\ast}+\alpha I)^{\frac{-2i-1}{2}}Q_my^{\delta}\|^2$, we obtain
\begin{align*}
  \frac{de(\alpha)}{d\alpha}=&i\alpha^{2i-1}\|(T_{h,m}T_{h,m}^{\ast}+\alpha I)^{\frac{-2i-1}{2}}Q_my^{\delta}\|^2\\
  &+i\biggl\langle\int_{0}^{\infty}\frac{\alpha^{i-1}}{(\mu+\alpha)^{\frac{1}{2}}}dF_{\mu}^{h,m}Q_m(y-y^{\delta}),(T_{h,m}T_{h,m}^{\ast}+\alpha I)^{\frac{-2i-1}{2}}Q_my^{\delta}\biggr\rangle,
\end{align*}
and collecting 
$i\|(T_{h,m}T_{h,m}^{\ast}+\alpha I)^{\frac{-2i-1}{2}}Q_my^{\delta}\|\eqqcolon K$ from the
two terms gives
\begin{align*}
  \frac{de(\alpha)}{d\alpha}\geq K\left(\alpha^{2i-1}\|(T_{h,m}T_{h,m}^{\ast}+\alpha I)^{\frac{-2i-1}{2}}Q_my^{\delta}\|-\biggl\|\int_{0}^{\infty}\frac{\alpha^{i-1}}{(\mu+\alpha)^{\frac{1}{2}}}dF_{\mu}^{h,m}Q_m(y-y^{\delta})\biggr\|\right).
\end{align*}
The proposition now follows from
\begin{align*}
  \biggl\|\int_{0}^{\infty}\frac{\alpha^{i-1}}{(\mu+\alpha)^{\frac{1}{2}}}dF_{\mu}^{h,m}Q_m(y-y^{\delta})\biggr\|\leq\alpha^{i-\frac{3}{2}}\delta.
\end{align*}
\end{proof}

\begin{rmk}
This analysis is an addendum to the work of \cite[Appendix A]{bianchi2023itat}. Here equation \eqref{cond2} and condition \eqref{condtau} are a special case of \cite[equation (A.27)]{bianchi2023itat} and condition \cite[equation~(A.26)]{bianchi2023itat} respectively, under the new assumption $Q_mTx^{\dagger}=T_{h,m}x^{\dagger}$.
\end{rmk}


\subsection{More details on the convergence analysis}

In this section we provide  more insights on some of the results presented in \cite{bianchi2023itat}. Consider Assumptions~\ref{assumptions} along with the following conditions.

\begin{ass}
Let $W_m\subset\mathcal{Y}$, $T_h\in \mathcal{L}(\mathcal{X},\mathcal{Y})$, $\delta\geq 0$, and 
$y^{\delta}\in\mathcal{Y}$ be such that
\begin{equation}\label{conditions}\tag{C}
\begin{split}
    W_m\subset\overline{\Range(T)},\qquad &\Range(Q_mT_h)=W_m,\\
    \|Q_m(T-T_h)\|\leq h,\qquad &\|Q_m(y-y^{\delta})\|\leq\delta.
\end{split}
\end{equation}
\end{ass}

We define  the operator $T_m\coloneqq Q_mT$ and recall the well known notation for the source set
\begin{equation*}
    \mathcal{X}_{\nu,\rho}\coloneqq\lbrace x\in\mathcal{X}\mid x=(T^{\ast}T)^{\nu}w, w\in\ker(T)^{\perp} \text{ and }\|w\|\leq\rho\rbrace.
\end{equation*}

We first present a proof for an improved version of \cite[Lemma A.3]{bianchi2023itat}; see Remark~\ref{finalrmk}.

\begin{appendixlemma}\label{lemlimv1}
Let $x^{\dagger}\in\mathcal{X}_{\nu,\rho}$ and let $w$ be fixed. For $\nu\in[0,1)$ defining
\begin{equation*}
 b(\alpha,m)\coloneqq \alpha^{i(1-\nu)}\|(T_m^{\ast}T_m+\alpha I)^{-i}
 (T_m^{\ast}T_m)^{\nu}w\|,
\end{equation*}
it holds that 
\begin{equation*}
    \lim_{\substack{m\rightarrow\infty \\ \alpha\rightarrow 0^{+}}}b(\alpha,m)=0
\end{equation*}
if 
\begin{equation}\label{condplus}
\alpha=o\Bigl(\mu_{m}^{\frac{i-\nu}{i(1-\nu)}}\Bigr),
\end{equation}
when $\nu>0$ and $i>1$ for $\mu_{m}\coloneqq\min(\sigma(T^{\ast}_mT_m)\setminus\lbrace 0\rbrace)$, where $\sigma(T_m^{\ast}T_m)$ denotes the spectrum of $T_m^{\ast}T_m$.
\end{appendixlemma}
\begin{proof}
The case $i=1$ is treated in \cite[Lemma 5.5]{engl1996}. We only consider the case $i>1$.

If $\nu=0$, from the inequality
\begin{equation*}
    \alpha^{i}\|(T_m^{\ast}T_m+\alpha I)^{-i}w\|\leq\alpha\|(T_m^{\ast}T_m+\alpha I)^{-1}w\|
\end{equation*}
the result again follows from \cite[Lemma 5.5]{engl1996}.

Let $\nu>0$. Firstly note that since it holds $\mathcal{X}=\ker(T_m)\oplus\ker(T_m)^{\perp}$ for any $m$, by definition of $b(\alpha,m)$ we have the equality
\begin{equation*}
    \alpha^{i(1-\nu)}\|(T_m^{\ast}T_m+\alpha I)^{-i}
 (T_m^{\ast}T_m)^{\nu}w\|=\alpha^{i(1-\nu)}\|(T_m^{\ast}T_m+\alpha I)^{-i}
 (T_m^{\ast}T_m)^{\nu}w_m\|
\end{equation*}
where $w_m$ is the orthogonal projection of $w$ into $\ker(T_m)^{\perp}\subseteq\ker(T)^{\perp}$, in particular $\|w_m\|\leq\|w\|$. Let now $\lbrace E_{\mu}^{m}\rbrace_{\mu\in\R}$ be a spectral family for $T^{\ast}_mT_m$. We have the equality 
\begin{equation*}
  b^{2}(\alpha,m)=\int_0^{\infty}\frac{\alpha^{2i(1-\nu)}}{(\mu+\alpha)^{2i(1-\nu)}}\frac{\mu^{2\nu}}{(\mu+\alpha)^{2\nu}}dE_{\mu}^{m}\|w_m\|^2.
\end{equation*}
By H\"older's inequality,
\begin{align*}
 b^2(\alpha,m)&\leq\left[\int_0^{\infty}\left(\frac{\alpha}{\mu+\alpha}\right)^{2i}dE_{\mu}^{m}\|w_m\|^2\right]^{1-\nu}\left[\int_0^{\infty}\left(\frac{\mu}{(\mu+\alpha)^{i}}\right)^{2}dE_{\mu}^{m}\|w_m\|^2\right]^{\nu}\\
 &\leq\left[\int_0^{\infty}\left(\frac{\alpha}{\mu+\alpha}\right)^{2i}dE_{\mu}^{m}
 \|w_m\|^2\right]^{1-\nu}\mu_{m}^{2\nu(1-i)}\|w_m\|^{2\nu}\\
 &=\alpha^{2i(1-\nu)}\|(T_m^{\ast}T_m+
 \alpha I)^{-i}w_m\|^{2(1-\nu)}\mu_{m}^{2\nu(1-i)}\|w_m\|^{2\nu}\\
 &\leq\alpha^{2i(1-\nu)}\frac{1}{(\mu_m+\alpha)^{2i(1-\nu)}}\mu_m^{2\nu(1-i)}\|w_m\|^2\\
 &\leq\frac{\alpha^{2i(1-\nu)}}{\mu_m^{2(i-\nu)}}\|w_m\|^2\leq\frac{\alpha^{2i(1-\nu)}}{\mu_m^{2(i-\nu)}}\|w\|^2\rightarrow 0.
\end{align*}
where in the end we used hypothesis \eqref{condplus}.
\end{proof}

Note that equation~\eqref{condplus} is trivially satisfied when $T^{\ast}T$ has finite spectrum, which is the case in the finite dimensional setting.

We now prove \cite[Proposition A.4]{bianchi2023itat} using slightly different hypothesis; see Remark~\ref{finalrmk} and Remark~\ref{secondfinalrmk}.

\begin{appendixproposition}\label{propcorrection}
Let Conditions~\ref{conditions} be satisfied. Moreover, let 
$x^{\dagger}\in\mathcal{X}_{\nu,\rho}$ so that $x^{\dagger}=(T^{\ast}T)^{\nu}w$. For $m$ sufficiently large, there is a unique
$\alpha>0$ that solves
\begin{equation}\label{eqtosolve}
 \alpha^{\frac{2i+1}{2}}\|(T_m^{\ast}T_m+\alpha I)^{-i}(T_m^{\ast}T_m)^{\nu}w\|=
 (i(i+1)h\|x^{\dagger}\|+i\delta)/2.
\end{equation}
Moreover, if $\nu=1$ and $i>1$ assume that $w=Q_ww$ for some $Q_w$ orthogonal projector into a finite dimensional eigenspace of $T^{\ast}T$.

Then, it holds that $\alpha\rightarrow 0$ for $h,\delta\rightarrow 0$, $m\rightarrow\infty$, and, if moreover $\alpha=o(\mu_m^{\frac{i-\nu}{i(1-\nu)}})$ when $\nu<1$ and $i>1$, we have
\begin{equation*}
  \|x^{\dagger}-x_{\alpha,m,i}^{\delta,h}\|=\begin{cases}
      o(1)&\text{if $\nu=0$,}\\
      o((h+\delta)^{\frac{2i\nu}{2i\nu+1}})+O(\gamma_m^{2\nu}\|w\|)&\text{if $0<\nu<1$,}\\
      O((h+\delta)^{\frac{2i}{2i+1}})+O(\gamma_m\|(I-Q_m)Tw\|)&\text{if $\nu=1$,}
  \end{cases}
\end{equation*}
where $\gamma_m\coloneqq\|(I-Q_m)T\|$.
\end{appendixproposition}
\begin{proof}
The existence and uniqueness of a solution $\alpha$ for equation~\eqref{eqtosolve} follows since for $m$ large enough the left-hand side is a nonvanishing monotonically increasing function with image the non-negative real semi-axis. The limit $\alpha\to 0$ follows since the right-hand side is monotonically decreasing as $\delta,h\rightarrow 0$.

From Conditions~\ref{conditions} used in \cite[Lemma A.2]{bianchi2023itat}, the choice of $\alpha$, and the fact that $x^{\dagger}\in\mathcal{X}_{\nu,\rho}$, it follows
\begin{align*}
    \|x^{\dagger}-x_{\alpha,m,i}^{\delta,h}\|
    &\leq\alpha^i\bigl\|(T^{\ast}_mT_m+\alpha I)^{-i}x^{\dagger}\bigr\|
      +\frac{i(i+1)h\|x^{\dagger}\|+i\delta}{2\sqrt{\alpha}}\\
    &=\alpha^i\bigl\|(T^{\ast}_mT_m+\alpha I)^{-i}(T^{\ast}T)^{\nu}w\bigr\|
      +\alpha^{i}\bigl\|(T_m^{\ast}T_m+\alpha I)^{-i}(T_m^{\ast}T_m)^{\nu}w\bigr\|\\
    &\leq 2\alpha^{i}\bigl\|(T_m^{\ast}T_m+\alpha I)^{-i}(T_m^{\ast}T_m)^{\nu}w\bigr\|
      +\bigl\|[(T^{\ast}T)^{\nu}-(T_m^{\ast}T_m)^{\nu}]w\bigr\|,
\end{align*}
where we also used the inequality $\alpha^{i}\|(T_m^{\ast}T_m+\alpha I)^{-i}\|\leq 1$ to obtain the summand $\|[(T^{\ast}T)^{\nu}-(T_m^{\ast}T_m)^{\nu}]w\|$. The estimates for this term follow from \cite[Lemma 2.6]{king1988} for all cases.

If $\nu=0$, then the assertion follows directly from Lemma~\ref{lemlimv1} (using its $\nu=0$ case together with \cite[Lemma 5.5]{engl1996}).

Let now $0<\nu\leq 1$. Set
\begin{equation*}
    b(\alpha,m)\coloneqq \alpha^{i(1-\nu)}
  \bigl\|(T_m^{\ast}T_m+\alpha I)^{-i}(T_m^{\ast}T_m)^{\nu}w\bigr\|
\end{equation*}
as in Lemma~\ref{lemlimv1}, and denote
\begin{equation*}
    R(h,\delta)\coloneqq \frac{i(i+1)h\|x^{\dagger}\|+i\delta}{2}.
\end{equation*}
Then \eqref{eqtosolve} reads
\begin{equation*}
    \alpha^{\frac{2i+1}{2}}
  \bigl\|(T_m^{\ast}T_m+\alpha I)^{-i}(T_m^{\ast}T_m)^{\nu}w\bigr\|
  = R(h,\delta).
\end{equation*}
Using the definition of $b(\alpha,m)$, we obtain
\begin{equation*}
    \alpha^{\frac12+i\nu} b(\alpha,m)=R(h,\delta),
\end{equation*}
and hence
\begin{equation*}
   \alpha^{i}\bigl\|(T_m^{\ast}T_m+\alpha I)^{-i}(T_m^{\ast}T_m)^{\nu}w\bigr\|
  = \alpha^{i\nu}b(\alpha,m)
  = b(\alpha,m)^{\frac{1}{2i\nu+1}}\,
    R(h,\delta)^{\frac{2i\nu}{2i\nu+1}}. 
\end{equation*}
If $0<\nu<1$, by Lemma~\ref{lemlimv1} and the assumption $\alpha=o\bigl(\mu_m^{\frac{i-\nu}{i(1-\nu)}}\bigr)$ for the case $i>1$, we have
$b(\alpha,m)\to 0$ as $h,\delta\to0$ and $m\to\infty$, so
\begin{equation*}
  \,b(\alpha,m)^{\frac{1}{2i\nu+1}}\,
    R(h,\delta)^{\frac{2i\nu}{2i\nu+1}}
  = o\Bigl(R(h,\delta)^{\frac{2i\nu}{2i\nu+1}}\Bigr)
  = o\bigl((h+\delta)^{\tfrac{2i\nu}{2i\nu+1}}\bigr).
\end{equation*}
If $\nu=1$, from the hypothesis $w=Q_ww$ we get the bound
\begin{equation*}
    b(\alpha,m)=\|(T_m^{\ast}T_m+\alpha)^{-i}T_m^{\ast}T_mw\|=\|(T_m^{\ast}T_m+\alpha)^{-i}T_m^{\ast}T_mQ_ww\|\leq\|w\|/\mu_w^{i-1},
\end{equation*}
for $\mu_w\coloneqq\min(\sigma(T_w^{\ast}T_w)\setminus\lbrace 0\rbrace)$ where $T_w\coloneqq TQ_w$. Therefore,
\begin{equation*}
  \,b(\alpha,m)^{\frac{1}{2i\nu+1}}\,
    R(h,\delta)^{\frac{2i\nu}{2i\nu+1}}
  = O\Bigl(R(h,\delta)^{\frac{2i\nu}{2i\nu+1}}\Bigr)
  = O\bigl((h+\delta)^{\tfrac{2i\nu}{2i\nu+1}}\bigr).
\end{equation*}
\end{proof}

In the end, the result of \cite[Theorem A.7]{bianchi2023itat}, from which the convergence rates for both the iAT and the iGKT methods followed. Similarly as before, we state a version with slightly different hypothesis, see Remark~\ref{finalrmk} and Remark~\ref{secondfinalrmk}.

\begin{appendixtheorem}\label{conv2}
Let $C>1$ and $E>\|x^{\dagger}\|$, and let $T_{h,m}\coloneqq Q_mT_h$. 
Assume that Conditions~\ref{conditions} hold and that
\begin{equation*}
    Eh+C\delta\leq\|Q_my^{\delta}\|.
\end{equation*}
Moreover, let $x^{\dagger}\in\mathcal{X}_{\nu,\rho}$ so that $x^{\dagger}=(T^{\ast}T)^{\nu}w$ and, if $\nu=1$ and $i>1$ assume that $w=Q_ww$ for some $Q_w$ orthogonal projector into a finite dimensional eigenspace of $T^{\ast}T$.

Let $\alpha>0$ be the unique 
solution of
\begin{equation}\label{eq:param-hm}
 \alpha^{2i+1}\bigl\|(T_{h,m}T_{h,m}^{\ast}+\alpha I)^{-2i+1}Q_my^{\delta}\bigr\|^{2}=
 (Eh+C\delta)^2.
\end{equation}
Then the same asymptotic estimates of Proposition~\ref{propcorrection} hold for this parameter choice if
\begin{equation}\label{condplus-again}
  \alpha=o\Bigl(\mu_{m}^{\frac{i-\nu}{i(1-\nu)}}\Bigr)
\end{equation}
in the case $0<\nu<1$ and $i>1$ (with $\mu_m$ as in Lemma~\ref{lemlimv1}), .
\end{appendixtheorem}
\begin{proof}
As reported in \cite{bianchi2023itat}, the proof is identical to the proof of \cite[Proposition 3.4]{neubauer1988a} thanks to \cite[Lemma A.6]{bianchi2023itat}. In particular, the convergence rates are a consequence of Proposition~\ref{propcorrection}. The absence of the constant $(2i+1)$ in the lower bound for $E$ follows by applying \cite[Proposition A.8]{bianchi2023itat}, which assures the same convergence rates and thus the non-necessity of this constant.
\end{proof}

\begin{rmk}\label{finalrmk}
We compare the requirement \eqref{condplus}, which we use to show Lemma \ref{lemlimv1}, to the condition 
\begin{equation}\label{xx}
\alpha=o(\mu_m^{\frac{2\nu(i-1)}{1-\nu}})
\end{equation}
used in \cite[Proposition A.4, Theorem A.7]{bianchi2023itat}. Note that for $0<\nu<1$, the chain of inequalities
\begin{align*}
    &\frac{2\nu(i-1)}{1-\nu}>\frac{i-\nu}{i(1-\nu)},\\
    &2\nu i(i-1)>i-\nu,\\
    &\nu(2i^2-2i+1)>i,\\
    &\nu>\frac{i}{i^2+(i-1)^2}
\end{align*}
holds for $i$ large enough, from which it follows that in this case the requirement \eqref{condplus} is weaker than \eqref{xx}. For the case $\nu=0$ we only ask $\alpha=o(1)$ in both cases.
\end{rmk}

\begin{rmk}\label{secondfinalrmk}
Differently from \cite{bianchi2023itat}, to obtain the convergence rates stated in Proposition~\ref{propcorrection} and Theorem~\ref{conv2} for the case $\nu=1$ and $i>1$, we explicitly require the additional assumption $w=Q_w w$, where $Q_w$ is an orthogonal projector onto a finite dimensional eigenspace of $T^{\ast}T$. This assumption, implicitly satisfied in the finite-dimensional case, ensures the claimed stronger convergence rates. In its absence, the convergence rate  reduces to the one provided by \cite[Proposition A.10]{bianchi2023itat}, namely, $O((h+\delta)^{\frac{2}{3}})+O(\gamma_m^2)$.
\end{rmk}


\subsection*{Acknowledgments}
We would like to thank the referees for comments.
D. Bianchi is supported by the Startup Fund of Sun~Yat-sen~University. M. Donatelli is partially supported by MIUR - PRIN 2022 N.2022ANC8HL and GNCS-INdAM Project CUP\_E53C24001950001.

\section*{Declarations}
\subsection*{Conflicts of interest} Not applicable.

\printbibliography

@article{gnnr,
title={Arnoldi decomposition, {GMRES}, and preconditioning for linear discrete ill-posed problems},
author={Gazzola, Silvia and Noschese, Silvia and Novati, Paolo and Reichel, Lothar},
journal={Applied Numerical Methods},
volume={142},
year={2019},
pages={102--121}
}

@article{bianchi2023itat,
  title={Convergence analysis and parameter estimation for the iterated {A}rnoldi-{T}ikhonov method},
  author={Bianchi, Davide and Donatelli, Marco and Furchi, Davide and Reichel, Lothar},
  journal={Numerische Mathematik},
  volume={157},
  year={2025},
  pages={749--779}
}

@article{IRTools,
title={{IR} {T}ools: a {MATLAB} package of iterative regularization methods and 
large-scale test problems},
author={Gazzola, Silvia and Hansen, Per Christian and Nagy, James G.},
journal={Numerical Algorithms},
volume={81},
pages={773--811},
year={2019}
}

@article{bjorck88,
  title={A bidiagonalization algorithm for solving large and sparse ill-posed systems of 
  linear equations}, 
  author={Bj\"orck, {\AA}ke},
  journal={BIT Numerical Mathematics},
  volume={18},
  year={1988},
  pages={659--670}
}

@book{bjorck24,
 title={Numerical Methods for Least Squares Problem},
  author={Bj\"orck, {\AA}ke},
  publisher ={SIAM, Philadelphia},
  year={2024}
}

@book{bjorck14,
title={Numerical Methods in Matrix Computations},
author={Bj\"orck, {\AA}ke},
publisher={Springer, New York},
year={2014}
}

@article{elden77,
title={Algorithms for the regularization of ill-conditioned least squares problems},
author={Eld\'en, Lars},
journal={BIT Numerical Mathematics},
volume={17},
year={1977},
pages={134--145}
}

@article{elden90,
author={Eld\'en, Lars},
title={Algorithms for the computation of functionals defined on the solution of a discrete ill-posed problem},
journal={BIT Numerical Mathematics},
volume={30},
year={1990},
pages={466--483}
}

@article{elden82,
title={A weighted pseudoinverse, generalized singular values and constraint least squares problems},
author={Eld\'en, Lars},
journal={BIT Numerical Mathematics},
volume={2},
year={1982},
pages={487--502}
}

@article{elden05,
author = {Eld\'en, Lars and Hansen, Per Christian and Marielba Rojas},
title = {Minimization of linear functionals defined on solutions of large-scale discrete ill-posed problems},
journal = {BIT Numerical Mathematics},
year = {2005},
pages={329--340}
}

@article{calvetti2003,
title={Tikhonov regularization of large linear problems}, 
author={Calvetti, D. and Reichel, L.},
journal={BIT Numerical Mathematics},
volume={43},
year={2003},
pages={263--283}
}

@article{greengard1997,
title={A new version of the fast multipole method for the {L}aplace equation in three 
dimensions},
author={Greengard, L. and Rokhlin, V.},
journal={Acta Numerica},
volume={6}, 
year={1997},
pages={229--269}
}

@inproceedings{calvetti2000,
title={Restoration of images with spatially variant blur by the {GMRES} method},
author={Calvetti, D. and Lewis, B. and Reichel, L.},
booktitle={Advanced Signal Processing Algorithms, Architectures, and Implementations X, 
ed. F. T. Luk, Proceedings of the Society of Photo-Optical Instrumentation Engineers 
(SPIE)},
volume={4116}, 
year={2000},
pages={364--374}
}

@article{buccini2020gsvd,
title={Generalized singular value decomposition with iterated {T}ikhonov regularization},
author={Buccini, Alessandro and Pasha, Mirjeta and Reichel, Lothar},
journal={Journal of Computational and Applied Mathematics},
volume={373}, 
year={2020},
pages={112276}
}

@article{alqahtani2023error,
title={Error estimates for Golub-Kahan bidiagonalization with {T}ikhonov regularization
for ill-posed operator equations},
  author={Alqahtani, Abdulaziz and Ramlau, Ronny and Reichel, Lothar},
  journal={Inverse Problems},
year={2023},
vol={39},
pages={025002}
}

@article{calvetti2000tikhonov,
  title={Tikhonov regularization and the {L}-curve for large discrete ill-posed problems},
  author={Calvetti, D. and Morigi, S. and Reichel, L. and Sgallari, F.},
  journal={Journal of {C}omputational and {A}pplied {M}athematics},
  volume={123},
  number={1-2},
  pages={423--446},
  year={2000}
}

@article{lewis2009arnoldi,
  title={Arnoldi-{T}ikhonov regularization methods},
  author={Lewis, Bryan and Reichel, Lothar},
  journal={Journal of Computational and Applied Mathematics},
  volume={226},
  pages={92--102},
  year={2009}
}

@article{alkilayh2023arnoldi,
  title={Some numerical aspects of {A}rnoldi-{T}ikhonov regularization},
  author={Alkilayh, Maged and Reichel, Lothar},
  journal={Applied Numerical Mathematics},
  volume={185},
  pages={503--515},
  year={2023}
}

@book{golub2013,
title={Matrix Computations, 4th ed.},
author={Golub, G. H. and Van Loan, C. F.},
publisher={Johns Hopkins University Press, Baltimore},
year={2013}
}

@book{engl1996,
title={Regularization of Inverse Problems},
author={Engl, H. W. and Hanke, M. and Neubauer, A.},
publisher={Kluwer, Dordrecht},
year={1996}
}

@book{natterer2001,
    title      ={The Mathematics of Computerized Tomography},
	author     ={Natterer, F.},
	publisher  ={SIAM, Philadelphia},
	year       ={2001}
}

@book{saad2003,
    title      ={Iterative Methods for Sparse Linear Systems},
	author     ={Saad, Y.},
	edition    ={2nd ed.},
	publisher  ={SIAM, Philadelphia},
	year       ={2003}
}

@book{de1978practical,
	title      ={A {P}ractical {G}uide to {S}plines},
	author     ={de Boor, C.},
    publisher  ={Springer, New York},
	year       ={1978}
}

@book{goodman2016discrete,
title={Discrete {F}ourier and {W}avelet {T}ransforms: {A}n {I}ntroduction 
through {L}inear {A}lgebra with {A}pplications to {S}ignal {P}rocessing},
author={Goodman, R. W.},
publisher={World Scientific Publishing Company, London},
year={2016}
	
}

@article{neubauer1988a,
    title      ={An a posteriori parameter choice for {T}ikhonov regularization in the presence of modeling error},
	author     ={Neubauer, A.},
	journal    ={Applied Numerical Mathematics},
	volume     ={4},
	pages      ={507--519},
    year       ={1988}
}

@article{king1988,
title ={A variant of finite-dimensional {T}ikhonov regularization with a-posteriori parameter choice},
    author     ={King, J. T. and Neubauer, A.},
	journal    ={Computing},
	volume     ={40},
	pages      ={91--109},
    year       ={1988}
}

@article{ramlau2019,
    title      ={Error estimates for {A}rnoldi-{T}ikhonov regularization for ill-posed operator equations},
	author     ={Ramlau, R. and Reichel, L.},
	journal    ={Inverse Problems},
	volume     ={35},
	pages      ={055002},
    year       ={2019}
}

@article{natterer1977,
    title      ={Regularization of ill-posed problems by projection methods},
	author     ={Natterer, F.},
	journal    ={Numererische Mathematik},
	volume     ={28},
    pages      ={329--341},
    year       ={1977}
}

@article{raffetseder2016,
    title      ={Optimal mirror deformation for multi-conjugate adaptive optics systems},
	author     ={Raffetseder, S. and Ramlau, R. and Yudytski, M.},
	journal    ={Inverse Problems},
	volume     ={32},
	pages      ={025009},
    year       ={2016}
}

@article{bentbib2018,
    title      ={Solution methods for linear discrete ill-posed problems for color image restoration},
	author     ={Bentbib, A. H. and El Guide, M. and Jbilou, K. and Onunwor, E. and Reichel, L.},
	journal    ={BIT Numerical Mathematics},
	volume     ={58},
	pages      ={555--578},
    publisher  ={Springer},
    year       ={2018}
}

@article{diazdealba2019,
    title      ={Recovering the electrical conductivity of the soil via a linear integral model},
	author     ={D\'iaz de Alba, P. and Fermo, L. and van der Mee, C. and Rodriguez, G.},
	journal    ={Journal of Computational and Applied Mathematics},
	volume     ={352},
	pages      ={132--145},
    publisher  ={Elsevier},
    year       ={2019}
}

@article{ramlau2012,
    title      ={An efficient solution to the atmospheric turbulence tomography problem using {K}aczmarz iteration},
	author     ={Ramlau, R. and Rosensteiner, M.},
	journal    ={Inverse Problems},
	volume     ={28},
	pages      ={095004},
    publisher  ={IOP publishing},
    year       ={2012}
}

@article{hanke1998nonstationary,
  title={Nonstationary iterated Tikhonov regularization},
  author={Hanke, Martin and Groetsch, Charles W},
  journal={Journal of Optimization Theory and Applications},
  volume={98},
  pages={37--53},
  year={1998},
  publisher={Springer}
}

@article{buccini2017iterated,
  title={Iterated {T}ikhonov regularization with a general penalty term},
  author={Buccini, Alessandro and Donatelli, Marco and Reichel, Lothar},
  journal={Numerical Linear Algebra with Applications},
  volume={24},
  number={4},
  pages={e2089},
  year={2017}
}

@article{donatelli2012nondecreasing,
  title={On nondecreasing sequences of regularization parameters for nonstationary iterated {T}ikhonov},
  author={Donatelli, Marco},
  journal={Numerical Algorithms},
  volume={60},
  pages={651--668},
  year={2012},
  publisher={Springer}
}

@article{donatelli2015arnoldi,
  title={{A}rnoldi methods for image deblurring with anti-reflective boundary conditions},
  author={Donatelli, Marco and Martin, David and Reichel, Lothar},
  journal={Applied Mathematics and Computation},
  volume={253},
  pages={135--150},
  year={2015}
}

@article{gazzola2015krylov,
  title={On {K}rylov projection methods and {T}ikhonov regularization},
  author={Gazzola, Silvia and Novati, Paolo and Russo, Maria Rosaria},
  journal={Electronic Transactions on Numerical Analysis},
  volume={44},
  pages={83--123},
  year={2015}
}

@article{buccini2023arnoldi,
  title={An {A}rnoldi-based preconditioner for iterated {T}ikhonov regularization},
  author={Buccini, Alessandro and Onisk, Lucas and Reichel, Lothar},
  journal={Numerical Algorithms},
  volume={92},
  number={1},
  pages={223--245},
  year={2023},
  publisher={Springer}
}

@article{bianchi2015iterated,
  title={Iterated fractional Tikhonov regularization},
  author={Bianchi, Davide and Buccini, Alessandro and Donatelli, Marco and Serra-Capizzano, Stefano},
  journal={Inverse Problems},
  volume={31},
  number={5},
  pages={055005},
  year={2015},
  publisher={IOP Publishing}
}

@article{ps1982,
  title={LSQR: An Algorithm for sparse linear equations and sparse least squares},
  author={Paige, Christopher C. and Saunders, Michael A.},
  journal={Transactions on Mathematical Software},
  volume={8},
  number={1},
  pages={43--71},
  year={1982},
  publisher={ACM}
}

@article{os1981,
  title={A bidiagonalization-regularization procedure for large scale discretizations of ill-posed problems},
  author={O’Leary, Dianne P. and Simmons, John A.},
  journal={SIAM Journal on Scientific and Statistical Computing},
  volume={2},
  pages={474--489},
  year={1981},
  publisher={SIAM}
}

@article{gn2016,
  title={Inheritance of the discrete Picard condition in Krylov subspace methods},
  author={Gazzola, Silvia and Novati, Paola},
  journal={BIT Numerical Mathematics},
  volume={56},
  pages={893--918},
  year={2016},
  publisher={Springer}
}

@article{n2018,
  title={A convergence result for some Krylov-Tikhonov methods in Hilbert spaces},
  author={Novati, Paola},
  journal={Numerical Functional Analysis and Optimization},
  volume={39},
  pages={655--666},
  year={2018},
  publisher={Taylor & Francis}
}
\end{document}